\documentclass[reqno,centertags, 11pt]{amsart}

\usepackage{amssymb,amsmath,amsfonts,amssymb}%numbysec}
\usepackage{hyperref}
\usepackage{enumerate}

-\marginparwidth \oddsidemargin -9mm \evensidemargin \oddsidemargin

\usepackage{latexsym}
\advance\hoffset by 5mm

\def\@abssec#1{\vspace{.05in}\footnotesize \parindent .2in
{\bf #1. }\ignorespaces}
\newtheorem{theorem}{Theorem}[section]

\newtheorem{lemma}[theorem]{Lemma}
\newtheorem{proposition}[theorem]{Proposition}

\newtheorem{definition}[theorem]{Definition}
\newtheorem{remark}[theorem]{Remark}

\allowdisplaybreaks\numberwithin{equation}{section}

\title[Well-posedness for the 2D dynamics of dislocation densities]
{On the well-posedness of a 2D nonlinear and nonlocal system arising from the dislocation dynamics}
\author{Dong Li}
\address{Department of Mathematics, University of Iowa, 14 MacLean Hall, Iowa City, IA52242.}
\email{mpdongli@gmail.com}
\author{Changxing Miao}
\address{Institute of Applied Physics and Computational Mathematics, P.O. Box 8009, Beijing 100088, P.R. China.}
\email{miao\_{}changxing@iapcm.ac.cn}
\author{Liutang Xue}
\address{The Graduate School of China Academy of Engineering Physics, P.O. Box 2101, Beijing 100088, P.R. China.}
\email{xue\_{}lt@163.com}

\date{}

\begin{document}
%%%%%%%%%%%%%%%%%%%%%%%%%%%%%%%%%%%%%%%%%%%%%%%%%%%%%%%%%%%%%%%%%%%%%%%%%%%%%%%%%%%%%%%%%%%%%%%%%%%%
%abstract
\begin{abstract}
  In this paper we consider a 2D nonlinear and nonlocal model describing the dynamics of the dislocation densities.
We prove the local well-posedness of strong solution to this system in the suitable functional framework,
and we show the global well-posedness for some dissipative cases by the method of nonlocal maximum principle.
\end{abstract}

\maketitle
%%%%%%%%%%%%%%%%%%%%%%%%%%%%%%%%%%%%%%%%%%%%%%%%%%%%%%%%%%%%%%%%%%%%%%%%%%%%%%%%%%%%%%%%%%%%%%%%%%%

\section{Introduction}
\setcounter{section}{1}\setcounter{equation}{0}

In the materials science, dislocations are termed as certain defects shown by real crystals in the organization of their crystalline structure.
They were considered as the principal explanation of plastic deformation at the microscopic scale of materials.
Dislocations can move under the effect of an exterior stress. In a particular case where the defects are parallel line
in the three-dimensional space, dislocations can be viewed as points in a plane by considering their cross-sections.
These dislocations are called ``edge dislocations" which move in the direction of the ``Burgers vector" which has a fixed direction
(cf. \cite{HiLo} for more physical description).

In this paper we focus on the following nonlinear and nonlocal system on $\mathbb{R}^2$ which arise from the dislocation dynamics
\begin{equation}\label{eq 1}
\begin{cases}
  \partial_t \rho^+ + u \cdot\nabla \rho^+ + \kappa |D|^\alpha \rho^+ =0, \qquad \alpha\in ]0,2], \\
  \partial_t \rho^- - u \cdot\nabla \rho^- + \kappa |D|^\alpha \rho^- = 0, \\
  u=\big(\mathcal{R}_1^2 \mathcal{R}_2^2 (\rho^+ -\rho^-),0\big), \\
  \rho^+|_{t=0}=\rho^+_0,\quad \rho^-|_{t=0}=\rho^-_0,
\end{cases}
\end{equation}
where $\kappa\geq 0$ is the viscosity coefficient, $\mathcal{R}_i\triangleq \partial_i/|D|$ ($i=1,2$, $\partial_i\triangleq \partial_{x_i}$)
is the usual Riesz transform and $|D|^\alpha$ is defined via the Fourier transform
\begin{equation*}
  \widehat{|D|^\alpha f}(\zeta)= |\zeta|^\alpha \widehat{f}(\zeta).
\end{equation*}
The inviscid case (i.e. $\kappa=0$) of \eqref{eq 1} is the model introduced by I. Groma and P. Balogh in \cite{Gro97,GroBal99}
where they consider two types of dislocations in the plane $(x_1,x_2)$.
Typically for a given velocity field, the dislocations of type $(+)$ propagate in the direction $+b$, with $b=(1,0)$ the Burgers vector,
while those of type $(-)$ propagate in the direction $-b$.
The terms $\rho^\pm$ are the plastic deformations in the material. The velocity vector field $u$
is the shear stress in the material, which solves the equation of elasticity (cf. \cite[Section 2]{CEMR08}).
Another closely related physical quantities are the derivatives of $\rho^\pm$ in the $x_1$-direction
$\partial_1\rho^\pm$, denoting by $\theta^\pm$, which represent the dislocation densities of type $(\pm)$.
Physically, $\theta^\pm$ are non-negative functions.
In terms of $\theta^\pm$, one can also formally rewrite the system \eqref{eq 1} as follows
\begin{equation}\label{eq 2}
\begin{cases}
  \partial_t \theta^+ + \partial_1(u_1\, \theta^+)
  + \kappa |D|^\alpha \theta^+ =0, \qquad \alpha\in ]0,2], \\
  \partial_t \theta^- - \partial_1 (u_1\, \theta^-)
  + \kappa |D|^\alpha \theta^- = 0, \\
  u_1=\mathcal{R}_1 \mathcal{R}_2^2 |D|^{-1}(\theta^+-\theta^-), \\
  \theta^+|_{t=0}=\theta^+_0,\quad \theta^-|_{t=0}=\theta^-_0.
\end{cases}
\end{equation}

In \cite{CEMR08}, Cannone et al considered the inviscid system \eqref{eq 1} with the initial data
\begin{equation}\label{eq Data}
  \rho^\pm(t=0,x_1,x_2)=\rho^\pm_0(x_1,x_2)= \bar\rho^\pm_0(x_1,x_2) + L x_1,\quad L\geq 0,
\end{equation}
where $\bar\rho^\pm_0(x_1,x_2)=\rho^{\pm,per}_0(x_1,x_2)$ and $\rho^{\pm,per}$ is a $1$-periodic function in $x=(x_1,x_2)$,
and by exploiting a fundamental entropy estimate satisfied by the dislocation densities,
the authors can show the global existence of a weak solution.
In \cite{Elha}, El Hajj proved that the inviscid model \eqref{eq 1} has a unique local-in-time solution with the initial data
\eqref{eq Data} prescribed on $\mathbb{R}^2$ and $\bar\rho^\pm_0(x_1,x_2)\in C^r(\mathbb{R}^2)\cap L^p(\mathbb{R}^2)$
with $r>1$ and $p\in]1,\infty[$. Note that $L$ may be chosen large enough so that $\partial_1 \rho^\pm_0\geq 0$
and this property for $\partial_1 \rho^\pm$ can be satisfied up to some positive time $T$ depending on $L$ and the initial data. For the study of more general dynamics of dislocation lines, we also refer to the works of \cite{AHLM,BCLM} and references therein
for some existence and uniqueness results.

In this article, in contrast with \cite{Elha}, we start with studying the system \eqref{eq 2} about the dislocation densities,
and then from the relation between $\theta^\pm$ and $\rho^\pm$, we go back to the system \eqref{eq 1} to give the meaning. The first result is the local well-posedness of the solution to the system \eqref{eq 2}.
\begin{theorem}\label{thm local}
  Let $\kappa\geq 0$, $\alpha\in ]0,2]$, $p\in ]1,2[$, $m>2$ and $(\theta^+_0,\theta^-_0)\in H^m(\mathbb{R}^2)\cap L^p(\mathbb{R}^2)$
be composed of real scalar functions. Then there exists
$T>0$ depending only on $\|\theta^\pm_0\|_{H^m \cap L^p}$ such that the system \eqref{eq 2} has a unique solution
$(\theta^+,\theta^-)\in C([0,T];H^m\cap L^p)$. Moreover, we have
$(\theta^+,\theta^-)\in C^1([0,T]; H^{m_0})$ with $m_0=\min\{m-1,m-\alpha\}$.

Besides, let $T^*>0$ be the maximal existence time of $(\theta^+,\theta^-) \in C([0,T^*[;H^m\cap L^p)$, then if $T^*<\infty$,
we necessarily have
\begin{equation}
  \int_0^{T^*} \|(\theta^+,\theta^-)(t)\|_{L^\infty} \mathrm{d}t =\infty,
\end{equation}
where we have used the notation that $\|(f,g)\|_X\triangleq\|f\|_X+\|g\|_X$ for some $f,g\in X$.
\end{theorem}

We also have some further properties of the solution.
\begin{proposition}\label{prop FurP}
  Let $\kappa\geq 0$, $\alpha\in ]0,2]$, $p\in ]1,2[$, $m>4$.
Suppose that
$(\theta^+,\theta^-)\in C([0,T^*[; H^m\cap L^p)$ is the corresponding maximal lifespan solution of the system \eqref{eq 2}
obtained in Theorem \ref{thm local}. Then the following statements hold true.
\begin{enumerate}[(1)]
  \item If $\theta^\pm_0$ are non-negative, then $\theta^\pm(t)$ are also non-negative for all $]0,T^*[$.

  \item Assume that $\kappa=0$ or $\kappa>0$ and $\alpha\in ]\frac{1}{2},2]$. If $\theta^\pm_0\in L^{\infty,1}_{x_2,x_1}(\mathbb{R}^2)$ (for
  definition see the next section) are non-negative, then
  $\theta^\pm\in L^\infty ([0,T^*[;L^{\infty,1}_{x_2,x_1})$ satisfies that
  \begin{equation}\label{eq MP}
    \|\theta^\pm(t)\|_{L^{\infty,1}_{x_2,x_1}}\leq \|\theta^\pm_0\|_{L^{\infty,1}_{x_2,x_1}},\quad \forall t\in [0,T^*[.
  \end{equation}
  Besides, the expression
  \begin{equation}\label{eq Exp}
    \rho^\pm(t,x_1,x_2)\triangleq \int_{-\infty}^{x_1} \theta^\pm(t,\tilde{x}_1,x_2)\mathrm{d}\tilde{x}_1,\quad
    \forall (t,x_1,x_2)\in [0,T^*[\times \mathbb{R}^2
  \end{equation}
  is well-defined and $\rho^\pm$ are the mild solutions to the system \eqref{eq 1}.

  \item If the conditions of (2) are supposed, and we moreover assume that for each $k=1,2,3$,
   $\partial_2^k \rho^\pm_0\in L_x^\infty(\mathbb{R}^2)$ and
   $\lim_{x_1\rightarrow-\infty}\partial_2^k \rho_0^\pm(x)=0$ for every $x_2\in\mathbb{R}$, then
   $$\rho^\pm\in L^\infty([0,T^*[, W^{3,\infty})\cap C([0,T^*[; W^{1,\infty}),$$ and
   $(\rho^+,\rho^-)$ satisfies the system \eqref{eq 1} in the classical pointwise sense.

  \item Under the assumption of (3), then for every $\epsilon>0$ and $t\in ]0,T^*[$, there exists $R>0$
  depending on $\kappa,\epsilon,t$ and $\|\theta^\pm\|_{L^\infty_t (H^m\cap L^p)}$ such that
  \begin{equation}\label{eq NRDecay}
    \|\nabla\rho^\pm\|_{L^\infty([0,t]; L^{\infty}_x (B_R^c))}\leq \|\nabla\rho^\pm_0\|_{L^\infty_x}+ \epsilon,
  \end{equation}
  where $B_R\triangleq \{x\in\mathbb{R}^2; \, |x|< R\}$ and $B_R^c$ is its complement.
\end{enumerate}

\end{proposition}

\begin{remark}
  Under the conditions of Proposition \ref{prop FurP}-(3), the corresponding solutions $\theta^\pm$ and $\rho^\pm$ are very locally well-posed,
and we only note that $\rho^\pm$ and $\partial_2\rho^\pm$ in general are bounded functions and don't satisfy the spatial decay property,
due to the physical constraint $\partial_1\rho^\pm\geq 0$.
Compared with those of \cite{Elha}, these initial data are of different type,
and they may have more advantage to guarantee the extension from the local solution to the global solution
(this can be convinced in some dissipative cases as follows).
We also notice that these assumptions can admit a large class of initial data, for instance,
the data of the form $\theta_0^\pm(x)=f^\pm(x_1) g^\pm(x_2)$
which satisfies that~$f^\pm\in H^m(\mathbb{R})\cap L^1(\mathbb{R})$,~$g^\pm\in H^m(\mathbb{R})\cap L^p(\mathbb{R})$~($m>4,p\in]1,2[$).~
\end{remark}

Next we shall consider the dissipative cases to show some global results.
From Theorem \ref{thm local}, in order to show the global well-posedness of the system \eqref{eq 2}, one should
prove that for every $T\in ]0,T^*[$, there is an upper bound of the quantity
$\int_0^T \|(\theta^+,\theta^-)(t)\|_{L^\infty}\mathrm{d}t $, or equivalently,
$\int_0^T \|(\partial_1\rho^+,\partial_1\rho^-)(t)\|_{L^\infty}\mathrm{d}t$. It seems very hard to obtain such a
bound directly from the system \eqref{eq 2}, thus here we shall turn to take advantage of the system \eqref{eq 1}
to give the desired bound.

Observe that for $\theta^-_0\equiv 0$, from the uniqueness issue in Theorem \ref{thm local}
and the fact that zero solution is a solution to the equation of $\theta^-$,
we have that $\theta^-(t)=\rho^-(t)\equiv 0$ for all $t\in[0,T^*[$.
By setting $\rho\triangleq \rho^+-\rho^-  =\rho^+$, we obtain
\begin{equation}\label{eq 4}
\begin{cases}
  \partial_t \rho + u\cdot\nabla \rho + \kappa |D|^\alpha \rho =0,  \quad \alpha\in ]0,2], \\
  u=(\mathcal{R}_1^2 \mathcal{R}_2^2\rho,0), \quad  \rho|_{t=0}=\rho_0.
\end{cases}
\end{equation}
The equation \eqref{eq 4} is reminiscent of the surface quasi-geostrophic (abbr. SQG) equation
\begin{equation}\label{eq SQG}
\begin{cases}
  \partial_t \rho + u\cdot\nabla \rho + \kappa |D|^\alpha \rho =0,\quad \alpha\in ]0,2], \\
  u=(-\mathcal{R}_2\rho,\mathcal{R}_1\rho), \quad  \rho|_{t=0}=\rho_0,
\end{cases}
\end{equation}
which arises from the geostrophic study of strongly rotating fluids (\cite{ConMT}) and has been intensely studied
in recent years (cf. \cite{CV,ChenMZ,ConWu99,CorC,Dab,KisNV,Kis,Res} and references therein).
For the dissipative (i.e. $\kappa>0$) SQG equation,
so far we only know that the cases of $\alpha\in[1,2]$ are global well-posed in various functional spaces,
and whether the supercritical cases of $\alpha\in]0,1[$ are global well-posed or not remains an outstanding open problem. We here briefly recall
some remarkable results. For the subcritical cases (i.e. $\alpha\in]1,2]$), it has been known that the SQG equation has
global strong solutions since the works \cite{Res} and \cite{ConWu99}. For the subtle critical case (i.e. $\alpha=1$),
the issue of global regularity was independently settled by \cite{KisNV} and \cite{CV} almost at the same time.
Kiselev et al in \cite{KisNV} proved the global well-posedness with the periodic smooth data by developing a new method
called the nonlocal maximum principle method, whose idea is to show that a family of suitable moduli of continuity are preserved by the evolution.
From a totally different direction, Caffarelli and Vasseur in \cite{CV}
established the global regularity of weak solutions by deeply exploiting the De Giorgi's iteration method.
We also refer to \cite{KisN} and \cite{ConV} for another two delicate and still quite different proofs of the same issue.

Compared to the SQG equation, the main disadvantage of the simplified model \eqref{eq 4} is that
the velocity field $u$ in \eqref{eq 4} is not divergence-free. This deficiency often leads to much difficulty in the application of the existing methods
(like Caffarelli-Vasseur's method), thus despite its possible advantage, we here do not expect to obtain better well-posed results than the SQG equation.
Hence, we hope that the coupling system \eqref{eq 1} in the cases of $\kappa>0$ (for brevity, setting $\kappa=1$) and
$\alpha\in [1,2]$ can generate a unique global strong solution and there is an upper bound
of the quantity $\int_0^T\|(\partial_1\rho^+,\partial_1\rho^-)(t)\|_{L^\infty}\mathrm{d}t$ for every $T\in ]0,T^*[$.
We find that the method of nonlocal maximum principle originated in \cite{KisNV} is not sensitive to the divergence-free condition of the velocity field,
and by applying this method, we indeed can prove the global results for the system \eqref{eq 1} in the cases $\alpha\in [1,2]$.
More precisely, we have
\begin{theorem}\label{thm glob}
  Let $\kappa=1$, $\alpha\in[1,2]$, $(\theta^+_0,\theta^-_0)$ be composed of non-negative real functions which belong to
$H^m(\mathbb{R}^2)\cap L^p(\mathbb{R}^2) \cap L^{\infty,1}_{x_2,x_1}(\mathbb{R}^2)$ with $m>4$, $p\in ]1,2[$.
Assume $\rho^\pm_0(x_1,x_2)=\int_{-\infty}^{x_1}\theta^\pm_0(\tilde{x}_1,x_2)\mathrm{d}
\tilde{x}_1$ satisfy that for each $k=1,2,3$, $\partial_2^k \rho^\pm_0\in L_x^\infty(\mathbb{R}^2)$ and
$\lim_{x_1\rightarrow-\infty}\partial_2^k \rho^\pm_0(x)=0$ for every $x_2\in\mathbb{R}$.
Then there exists a unique global solution
\begin{equation*}
  (\theta^+,\theta^-)\in C([0,\infty[;H^m\cap L^p)\cap L^\infty([0,\infty[;L^{\infty,1}_{x_2,x_1})
\end{equation*}
to the system \eqref{eq 2}. Moreover, $(\rho^+,\rho^-)\in L^\infty([0,\infty[; W^{3,\infty})\cap C([0,\infty[;W^{1,\infty})$ solves
the system \eqref{eq 1} in the classical pointwise sense.
\end{theorem}

Compared with the application of nonlocal-maximum-principle method to the SQG equation, there are another two noticeable different points:
the first is that what we considered here is a coupling system instead of a single equation, and
the second is that $(\rho^+,\rho^-)$ does not have the spatial decay property that
$\|(\nabla\rho^+,\nabla\rho^-)\|_{L^\infty([0,t]; L^\infty_x(B^c_R))}\rightarrow 0$ as $R\rightarrow \infty$ for each $t\in]0,T^*[$.
Notice that in the works \cite{AbidiH,DongD,MiaoX},
this spatial decay property is needed when applying the method of \cite{KisNV} to the whole-space SQG-type equation.
For the first point, we find that by proper modification in the scheme,
the nonlocal maximum principle method can still be suited to the system \eqref{eq 1}.
While for the second point, we observe that we indeed do not need such a strong decay property, and
what we need is that the Lipschitz norm of $(\rho^+,\rho^-)$ does not grow rapidly near infinity (cf. \eqref{eq key1}),
which just can be implied by Proposition \ref{prop FurP}-(4).

In the proof of Theorem \ref{thm glob}, Proposition \ref{prop FurP}-(2)(3) will also play an important role.
Since in the program of the nonlocal-maximum-principle method, we need that $(\rho^+,\rho^-)$ satisfies the system \eqref{eq 1} in the classical pointwise
sense and it also has sufficient smoothness property.

\begin{remark}\label{Rem1}
  From the direction of showing the regularity of weak solutions to the system \eqref{eq 1},
so far there is no direct result implying the global regularity,
due to that the velocity field $u=(\mathcal{R}_1^2\mathcal{R}_2^2(\rho^+-\rho^-),0)$ is neither divergence-free
nor belonging to $L^\infty_{t,x}$.
The main obstacle lies on the improvement from the bounded solution to the H\"older continuous solution;
as far as we know, the best result is as Silvestre \cite{Silv} shows, which calls for $u\in L^\infty_{t,x}$
to ensure that this improvement is satisfied for the drift-diffusion equation
$\partial_t \rho + u\cdot\nabla \rho + |D|^\alpha \rho =0$ with $\alpha\in[1,2[$ and general velocity field $u$.
\end{remark}

\begin{remark}
  The procedure in showing the global part of Theorem \ref{thm glob} can be applied to the Groma-Balogh model with
generalized dissipation, and we shall sketch it in the appendix.
\end{remark}

The paper is organized as follows. In Section \ref{sec prel}, we present some
preparatory results including some auxiliary lemmas and some facts about the modulus of continuity.
We show Theorem \ref{thm local}, Proposition \ref{prop FurP} and Theorem \ref{thm glob}
in Section \ref{sec local}--\ref{sec glob} respectively.

Throughout this paper, $C$ stands for a constant which may be different from line to line. For two quantities $X$ and $Y$, we
sometimes use $X\lesssim Y$ instead of $X\leq C Y$, and we use $X\approx Y$ if both $X\lesssim Y$ and $Y\lesssim X$ hold.
Denote $\widehat f$ the Fourier transform of $f$, i.e., $\widehat f(\zeta)=\int_{\mathbb{R}^2}e^{i x\cdot \zeta} f(x)\mathrm{d}\zeta$.

%%%%%%%%%%%%%%%%%%%%%%%%%%%%%%%%%%%%%%%%%%%%%%%%%%%%%%%%%%%%%%%%%%%%%%%%%%%%%%%%%%%%%%%%%%%%%%%%%%%%%%%%%%%%%%
\section{Preliminaries}\label{sec prel}
\setcounter{section}{2}\setcounter{equation}{0}
In this preparatory section, we compile the definitions of functional spaces used in this paper, some auxiliary lemmas
and some facts related to the modulus of continuity.

\subsection{Functional spaces and auxiliary lemmas}\label{subsec Pre1}

For $q\in [1,\infty]$, $L^q_x= L^q_x(\mathbb{R}^2)$, $L^q_{x_i}=L^q_{x_i}(\mathbb{R})$ ($i=1,2$) denote the usual Lebesgue spaces,
and we sometimes abbreviate $L^q_x(\mathbb{R}^2)$ by $L^q$. For $(q,r)\in [1,\infty]^2$, denote
$L^{q,r}_{x_2,x_1}=L^{q,r}_{x_2,x_1}(\mathbb{R}^2)$
the set of the tempered distributions $f\in\mathcal{S}'(\mathbb{R}^2)$ satisfying that
\begin{equation*}
  \|f\|_{L^{q,r}_{x_2,x_1}}\triangleq \|\|f(x)\|_{L^r_{x_1}}\|_{L^q_{x_2}}<\infty.
\end{equation*}
Similarly we can define the space $L^{q,r}_{x_1,x_2}=L^{q,r}_{x_1,x_2}(\mathbb{R}^2)$. Note that in general $L^{q,r}_{x_2,x_1}\neq L^{q,r}_{x_1,x_2}$.

For $s\in\mathbb{N}$, $q\in [1,\infty]$, $W^{s,q}=W^{s,q}(\mathbb{R}^2)$ denotes the usual Sobolev space:
\begin{equation*}
  W^{s,q}\triangleq \Big\{f\in\mathcal{S}'(\mathbb{R}^2);
  \|f\|_{W^{s,q}}\triangleq \sum_{|\beta|\leq s} \|\partial^\beta_x f\|_{L^q}<\infty \Big\},
\end{equation*}
When $q=2$, we also write $W^{s,2}=H^s=H^s(\mathbb{R}^2)$ with the norm $\|\cdot\|_{H^s}$.
For general $s\in\mathbb{R}$, we can also define the Sobolev space of fractional power $H^s=H^s(\mathbb{R}^2)$ via the Fourier transform, i.e.
\begin{equation*}
  H^{s}\triangleq\Big\{f\in \mathcal{S}'(\mathbb{R}^2);
  \|f\|_{H^s}\triangleq \| (1+|\zeta|^s)\widehat f(\zeta)\|_{L^2_\zeta}<\infty\Big\}.
\end{equation*}

In order to define the Besov spaces, we need the following dyadic partition of unity.
Let $\chi\in C^\infty(\mathbb{R}^2)$ be a radial function taking values in $[0,1]$, supported on the ball $B_{4/3}$
and $\chi\equiv 1$ on $B_1$. Define $\varphi(\zeta)= \chi(\zeta/2)-\chi(\zeta)$ for all $\zeta\in\mathbb{R}^2$, then
$\varphi$ is a smooth radial function supported on the shell $\{\zeta\in
\mathbb{R}^2: 1\leq |\zeta|\leq  \frac{8}{3} \}$. Clearly, \begin{equation*}
  \chi(\zeta)+\sum_{j\geq 0}\varphi(2^{-j}\zeta)=1, \quad
  \forall\zeta\in\mathbb{R}^2; \qquad
  \sum_{j\in \mathbb{Z}}\varphi(2^{-j}\zeta)=1, \quad \forall\zeta\neq 0.
\end{equation*}
Then for all $f\in\mathcal{S}'(\mathbb{R}^2)$, define the following nonhomogeneous Littlewood-Paley operators
\begin{equation*}
  \Delta_{-1}f \triangleq \chi(D)f; \qquad
  \Delta_j f \triangleq \varphi(2^{-j}D)f,\;\;\quad \forall j\in\mathbb{N},
\end{equation*}
and thus $\sum_{j\geq-1}\Delta_j f=f$. While for all $f\in\mathcal{S}'(\mathbb{R}^2)/\mathcal{P}(\mathbb{R}^2)$
with $\mathcal{S}'/\mathcal{P}$ the quotient space of tempered distributions up to polynomials,
define the homogeneous Littlewood-Paley operator
\begin{equation*}
  \dot{\Delta}_j f\triangleq \varphi(2^{-j}D)f,\qquad \forall j\in\mathbb{Z},
\end{equation*}
and thus $\sum_{j\in\mathbb{Z}}\dot\Delta_jf=f$.

Now for $(p,r)\in[1,\infty]^2$, $s\in\mathbb{R}$, we define the nonhomogeneous Besov space as follows
\begin{equation*}
  B^s_{p,r}\triangleq\Big\{f\in\mathcal{S}'(\mathbb{R}^2);\|f\|_{B^s_{p,r}}\triangleq\|\{2^{js}\|\Delta
  _j f\|_{L^p}\}_{j\geq -1}\|_{\ell^r}<\infty  \Big\}.
\end{equation*}
We point out that for all $s\in\mathbb{R}$, $B^s_{2,2}=H^s$. We also introduce the space-time Besov space
$L^\sigma([0,T],B^s_{p,r})$, abbreviated by $L^\sigma_{T}B^s_{p,r}$, which is the set of tempered distributions $f$ satisfying
\begin{equation*}
  \|f\|_{L^\sigma_T B^s_{p,r}}\triangleq \big\| \|\{2^{qs}\|\Delta_q f\|_{L^p_x}
  \}_{q\geq -1}\|_{\ell^r} \big\|_{L^\sigma_T }<\infty.
\end{equation*}

Bernstein's inequality is fundamental in the analysis involving frequency localized functions.
\begin{lemma}\label{lem Bern}
  Let $1\leq p\leq q\leq \infty$, $0<a<b<\infty$, $k\geq 0$, $\lambda>0$ and $f\in L^p(\mathbb{R}^2)$. Then,
\begin{align*}
  \textrm{if}\quad \mathrm{supp}\,\widehat{f}\subset\{\zeta: |\zeta|\leq \lambda b\},\;\Longrightarrow\;
  \||D|^k f\|_{L^q(\mathbb{R}^2)}\lesssim \lambda^{k+2(\frac{1}{p}-\frac{1}{q})}\|f\|_{L^p(\mathbb{R}^2)};
\end{align*}
and
\begin{align*}
  \textrm{if}\quad\mathrm{supp}\,\widehat{f}\subset  \{\zeta: a\lambda\leq |\zeta|\leq b\lambda\},\;\Longrightarrow\;
  \||D|^k f\|_{L^p(\mathbb{R}^2)}\approx \lambda^k \|f\|_{L^p(\mathbb{R}^2)}.
\end{align*}

\end{lemma}

We shall use the following lemma in the proof of the local existence.
\begin{lemma}\label{lem ExiKey}
  Let $f$ be a smooth real function on $\mathbb{R}^2$ and $u$ be a smooth vector field of $\mathbb{R}^2$.
Then the following assertions hold.
\begin{enumerate}[(1)]
\item For every $s\geq 0$, we have
\begin{equation}\label{eq ExiKey1}
\begin{split}
  & \sum_{j\geq 0} 2^{2js} \Big|\int_{\mathbb{R}^2}\Delta_j \big(\nabla \cdot(u\, f)\big)(x) \, \Delta_j f(x)\mathrm{d}x\Big|
  \lesssim\, \|\nabla u\|_{L^\infty} \|f\|_{ H^s}^2 + \|f\|_{L^\infty} \|\nabla u\|_{H^{s}} \|f\|_{H^s}.
\end{split}
\end{equation}

\item If $u=(\mathcal{R}_1\mathcal{R}_2^2 |D|^{-1} g,0)$, we have that for every $ s>1$ and $p\in ]1,2[$,
\begin{equation}\label{eq exiKey2}
  \sum_{j\in \mathbb{N}} \|\Delta_j \big(\nabla \cdot(u\, f) \big)\|_{L^p} \lesssim \|g\|_{H^s\cap L^p} \|f\|_{H^s}.
\end{equation}
\end{enumerate}
\end{lemma}

\begin{proof}[Proof of Lemma \ref{lem ExiKey}]
(1) The proof of \eqref{eq ExiKey1} essentially follows from the proof of \cite[Lemma 2.4]{LiRZ} with proper modification,
and here we omit the details.

(2) By Bony's decomposition, we get
\begin{equation*}
\begin{split}
  \sum_{j\in\mathbb{N}}\|\Delta_j \big(\nabla\cdot (u\, f) \big)\|_{L^p} = &
  \sum_{j\in\mathbb{N};|k-j|\leq 4} \|\Delta_j \big(\nabla\cdot(S_{k-1} u\, \Delta_k f)\big)\|_{L^p}
  + \sum_{j\in\mathbb{N}} \|\Delta_j \big(\nabla\cdot(\Delta_{-1}u\, S_1 f )\big)\|_{L^p} + \\
  & + \sum_{j\in\mathbb{N};k\geq j-4, k\in\mathbb{N}} \|\Delta_j\big(\nabla\cdot(\Delta_k u\, S_{k+2}f) \big)\|_{L^p} \\
  \triangleq & \,\mathrm{A}_1 +\mathrm{A}_2 +\mathrm{A}_3,
\end{split}
\end{equation*}
where $S_k =\sum_{-1\leq k'\leq k-1}\Delta_{k'}$ for every $k\in\mathbb{N}$. For $A_1$, from Bernstein's inequality, H\"older's inequality
and Hardy-Littlewood-Sobolev's inequality, we obtain
\begin{equation*}
\begin{split}
  A_1  & \lesssim \sum_{j\in \mathbb{N};|k-j|\leq 4} 2^j \|S_{k-1}u\|_{L^{2p/(2-p)}}\|\Delta_k f\|_{L^2} \\
  & \lesssim \|g\|_{L^p} \sum_{j\in\mathbb{N}} 2^{j(1-s)} 2^{k s} \|\Delta_k f\|_{L^2} \lesssim \|g\|_{L^p} \|f\|_{H^s}
\end{split}
\end{equation*}
For $A_2$, since $\Delta_j(\Delta_{-1}u\,S_1 f )= 0$ for $j\geq 3$, we get
\begin{equation*}
  A_2 \lesssim \sum_{0\leq j\leq 2}\|\Delta_j (\Delta_{-1}u\,S_1 f)\|_{L^p}\lesssim \|g\|_{L^p} \|f\|_{L^2}.
\end{equation*}
For $A_3$, from Bernstein's inequality, H\"older's inequality and Young's inequality, we have
\begin{equation*}
\begin{split}
  A_3 &\lesssim \sum_{j\in\mathbb{N}} \sum_{k\geq j-4,k\in\mathbb{N}} 2^j \|\Delta_k u\|_{L^{2p/(2-p)}} \|S_{k+2} f\|_{L^2} \\
  & \lesssim \|f\|_{L^2} \sum_{j\in\mathbb{N}} \sum_{k\geq j-4, k\in\mathbb{N}} 2^j 2^{-k}2^{k\frac{2(p-1)}{p}}\|\Delta_k g\|_{L^2} \\
  & \lesssim \|f\|_{L^2} \sum_{k\in\mathbb{N}} 2^{k\frac{2(p-1)}{p}}\|\Delta_k g\|_{L^2} \lesssim \|f\|_{L^2} \|g\|_{H^s}.
\end{split}
\end{equation*}
Gathering the upper estimates leads to \eqref{eq exiKey2}.
\end{proof}

The logarithmic inequality as follows will be used to show a refined blowup criterion.
\begin{lemma}\label{lem LogInq}
  Let $f\in H^m(\mathbb{R}^2)$ with $m>1$. Suppose that $S\in C^\infty (\mathbb{R}^2\setminus \{0\}) $ is a zero-order homogeneous function
and $\mathcal{T}$ is the operator on $\mathbb{R}^2$ with $S$ the symbol. Then we have
\begin{equation*}
  \|\mathcal{T}f\|_{L^\infty(\mathbb{R}^2)}\leq C+  C \|f\|_{L^\infty(\mathbb{R}^2)} \log\big(e + \|f\|_{H^m(\mathbb{R}^2)}\big).
\end{equation*}
\end{lemma}

\begin{proof}[Proof of Lemma \ref{lem LogInq}]
  By a high-low frequency decomposition, and from Bernstein's inequality and Calde\'ron-Zygmund's theorem, we have that for some $J\in\mathbb{N}$,
\begin{equation*}
\begin{split}
  \|\mathcal{T}f\|_{L^\infty} &\leq \Big(\sum_{j\leq -J} + \sum_{-J<j<J}+\sum_{j\geq J}\Big)
  \big( \|\dot \Delta_j\mathcal{T}f\|_{L^\infty}\big)\\
  & \lesssim \sum_{j\leq-J}2^j\|\dot\Delta_j\mathcal{T} f\|_{L^2} + \sum_{-J<j<J}\|\dot\Delta_j\mathcal{T}f\|_{L^\infty}
  +\sum_{j\geq J} 2^{j(1-m)}  2^{jm}\|\dot\Delta_j \mathcal{T} f\|_{L^2} \\
  & \lesssim 2^{-J} \|f\|_{L^2} + J \|f\|_{L^\infty} + 2^{-J(m-1)} \|f\|_{H^m} \\
  & \leq C 2^{-J a} \|f\|_{H^m} + CJ \|f\|_{L^\infty},
\end{split}
\end{equation*}
where $a\triangleq \min\{1,m-1\}$. Thus in order to make
$2^{-Ja}\|f\|_{H^m}\approx 1$, we can choose
\begin{equation*}
  J\triangleq [\log(e+\|f\|_{ H^m})/a]+1
\end{equation*}
with $[x]$ denoting the integer part of a real number $x$, and the desired estimate follows.
\end{proof}

We have the following integral expression of the operator $|D|^\alpha$ ($\alpha\in ]0,2[$) (cf. \cite[Theorem 1]{DroI}).
\begin{lemma}\label{lem DalpExp}
  Let $\alpha\in ]0,2[$, $r>0$ and $f\in C_b^2(\mathbb{R}^2)$ ($= C^2(\mathbb{R}^2)\cap W^{2,\infty}(\mathbb{R}^2)$). Then for every $x\in\mathbb{R}^2$,
\begin{equation*}
  |D|^\alpha f(x) = -\,c_\alpha \bigg(\int_{B_r}\frac{f(x+y)-f(x)-y\cdot \nabla f(x)}{|y|^{2+\alpha}} \mathrm{d}y
  + \int_{B_r^c} \frac{f(x+y)-f(x)}{|y|^{2+\alpha}}\mathrm{d}y\bigg),
\end{equation*}
where $c_\alpha = \frac{\alpha \Gamma(1+\alpha/2)}{2\pi^{1+\alpha}\Gamma(1-\alpha/2)}$ and $\Gamma$ is the usual Euler's function.
\end{lemma}

The following positivity lemma is also useful (cf. \cite[Lemma 2.7]{LiRZ}).
\begin{lemma}\label{lem Pos}
  Let $\kappa\geq 0$, $\alpha\in ]0,2]$, $p\in[1,\infty[$ and $T>0$. Denote $U_T\triangleq ]0,T]\times \mathbb{R}^2$,
and $C^{i,j}_{t,x}(U_T)\triangleq C^i_t(]0,T];C^j_x(\mathbb{R}^2))$, $i,j\in\mathbb{N}$.
Assume that $u\in C^{0,1}_{t,x}(U_T)$ is a real vector field of
$\mathbb{R}^2$, $\theta_0\in C(\mathbb{R}^2)$
is a real scalar and
$$\theta\in C^{1,0}_{t,x}(U_T)\cap C^{0,2}_{t,x}(U_T)
\cap C_{t,x}^0(\overline{U}_T)\cap L^p(U_T)$$
is a real scalar function satisfying the following pointwise inequality
\begin{equation*}
\begin{cases}
  \partial_t\theta + \nabla\cdot(u\, \theta) \geq - \kappa |D|^\alpha \theta, &\quad (t,x)\in U_T, \\
  \theta(0,x)=\theta_0(x),  & \quad x\in \mathbb{R}^2.
\end{cases}
\end{equation*}
We also suppose that there is a positive constant $C<\infty$ such that
\begin{equation*}
  \sup_{\overline{U}_T} |\theta| + \sup_{U_T} \big(|\partial_t \theta| + |\nabla \theta| + |\nabla^2\theta|\big)
  + \sup_{U_T} |\mathrm{div}\, u| \leq C,
\end{equation*}
Then if $\theta_0\geq 0$, we have $\theta\geq 0$ in $\overline{U}_T$.
\end{lemma}

\subsection{Modulus of continuity.}

We begin with introducing some terminology.
\begin{definition}
  A function $\omega:[0,\infty[ \mapsto [0,\infty[$ is called a modulus of continuity (abbr. MOC)
  if $\omega$ is continuous on $[0,\infty[$, increasing, concave, and
  piecewise $C^2$ with one-sided derivatives defined at each point in $[0,\infty[$ (maybe infinite at $\xi=0$).
  We call that a function $f:\mathbb{R}^2\rightarrow \mathbb{R}$ has (or obeys) the modulus of continuity $\omega$
  if $|f(x)-f(y)| \leq \omega(|x-y|)$ for every $x, y\in \mathbb{R}^2$. We also say that $f$ strictly obeys the modulus
  of continuity if the above inequality is strict for $x\neq y$.
\end{definition}

We first have the lemma concerning the action of the zero-order pseudo-differential operator like $\mathcal{R}_1^2\mathcal{R}_2^2$
on the function obeying MOC.
\begin{lemma}\label{lem MudRR}
Let $f,g: \mathbb{R}^2\mapsto\mathbb{R}$ obey the modulus of continuity $\omega$ and
the vector field $u=(\mathcal{R}_1^2\mathcal{R}_2^2 (f-g),0)$. Then the following assertions hold.
\begin{enumerate}[(1)]
  \item $u$ obeys the following modulus of continuity
  \begin{equation}\label{eq Omega}
    \Omega(\xi)= A_1 \omega(\xi) + A_2\Big(\int_0^\xi \frac{\omega(\eta)}{\eta}\mathrm{d}\eta
    + \xi \int_\xi^\infty \frac{\omega(\eta)}{\eta^2}\mathrm{d}\eta \Big),
  \end{equation}
  where $A_1$ and $A_2$ are positive absolute constants.
  \item If $f$ don't strictly have the MOC $\omega$ and there exists two separate points $x, y\in\mathbb{R}^2$ satisfying
  $f(x)-f(y)=\omega(\xi)$ with $\xi=|x-y|$. Then,
  \begin{equation}\label{eq MudRR1}
    |u\cdot\nabla f(x)- u\cdot\nabla f(y)|\leq \Omega(\xi)\omega'(\xi).
  \end{equation}
\end{enumerate}
\end{lemma}

\begin{proof}[Proof of Lemma \ref{lem MudRR}]
  (1) Since $m(\zeta)=\frac{\zeta_1^2\zeta_2^2}{|\zeta|^4}$ is the symbol of $\mathcal{R}_1^2 \mathcal{R}_2^2$ satisfying that
it is a zero-order homogeneous function belonging to $C^{\infty}(\mathbb{R}^2\setminus \{0\})$,
by virtue of \cite[Lemma 4.13]{Duo}, and denoting $\mathbb{S}^1$ the unit circle, we know that there exist $H\in C^{\infty}(\mathbb{S}^1)$
with zero average and two positive constants $a_1=\frac{1}{2\pi}\int_{\mathbb{S}^1}m(\zeta)\mathrm{d}\zeta$, $a_2>0$ such that
\begin{equation*}
  \mathcal{R}_1^2 \mathcal{R}_2^2 (f-g) = a_1\, (f-g) + a_2\,\Big(\textrm{p.v.} \, \frac{H(x')}{|x|^2}\Big)* (f-g),
\end{equation*}
with $x'\in\mathbb{S}^1$. Based on this expression and the fact that $f-g$ has the MOC $2\omega$, the desired result follows from
the deduction in \cite{KisNV} treating the corresponding point.

(2) We refer to \cite{KisNV} for the proof of this point.
\end{proof}

We also need a special action of the dissipation operator $|D|^\alpha$ on the function having MOC.
\begin{lemma}\label{lem MudDiss}
  Let $\alpha\in ]0,2]$, the real scalar function $f\in C^2_b(\mathbb{R}^2)$ obey the MOC $\omega$
but don't strictly obey it. Assume that there are two separate points $x,y\in\mathbb{R}^2$ such that
$f(x)-f(y)=\omega(\xi)$ with $\xi=|x-y|$. Then we have
\begin{equation*}
\begin{split}
  [-|D|^\alpha f](x)-[-|D|^\alpha f](y)\leq \Psi_\alpha(\xi),
\end{split}
\end{equation*}
where
\begin{equation}\label{eq Psi}
  \Psi_\alpha(\xi)=
  \begin{cases}
  B_\alpha \int_0^{\xi/2} \frac{\omega(\xi+2\eta)+
  \omega(\xi-2\eta)-2\omega(\xi)}{\eta^{1+\alpha}} \mathrm{d}\eta
  + B_\alpha \int_{\xi/2}^\infty \frac{\omega(\xi+2\eta) -\omega(2\eta-\xi)-2\omega(\xi)}{\eta^{1+\alpha}} \mathrm{d}\eta, &\quad \alpha\in]0,2[, \\
  2\omega''(\xi), & \quad \alpha=2,
  \end{cases}
\end{equation}
and $B_\alpha>0$.
\end{lemma}

The proof is essentially contained in \cite{KisNV,Kis}, and we omit the details here.

At last, we state a simple lemma concerning the function having MOC.
\begin{lemma}\label{lem MudLem}
Let $\omega$ be a MOC which in addition satisfies that
\begin{equation*}
  \omega(0)=0, \quad \omega'(0)<\infty,\quad \textrm{and} \;\; \omega''(0+)=-\infty.
\end{equation*}
If the real scalar function $f\in C^2_b(\mathbb{R}^2)$ obeys the MOC $\omega$, then for every $r\in ]0,\infty[$, we have
\begin{equation*}
  \|\nabla f\|_{L^\infty(B_r)} < \omega'(0).
\end{equation*}
\end{lemma}

\begin{proof}[Proof of Lemma \ref{lem MudLem}]
  The proof is similar to that in \cite{KisNV}. Indeed, since $|\nabla f|$ is a continuous function on $B_r$, we suppose that
it attains the maximum at $x\in \overline{B}_r$. Let $y= x+ \xi \ell$ with $\xi>0$ and
$\ell=\frac{\nabla f(x)}{|\nabla f(x)|}$, and by definition we have
$f(y)-f(x)\leq \omega(\xi)$. According to the Taylor formula, the left side of the inequality is bounded
from below by $|\nabla f(x)|\xi -\frac{1}{2}\|\nabla^2 f\|_{L^\infty}\xi^2$, while the right side
is bounded from above by $\omega'(0)\xi + g(\xi)\xi^2$ with $g(\xi)\rightarrow -\infty$ as $\xi\rightarrow 0+$.
Thus $|\nabla f(x)|\leq \omega'(0) + \xi(g(\xi)+ \frac{1}{2} \|\nabla^2 f\|_{L^\infty})$, and as $\xi$ small enough
the assertion follows.
\end{proof}

%%%%%%%%%%%%%%%%%%%%%%%%%%%%%%%%%%%%%%%%%%%%%%%%%%%%%%%%%%%%%%%%%%%%%%%%%%%%%%%%%%%%%%%%%%%%%%%%%%%%%%%%%%%%%%%%
\section{Proof of Theorem \ref{thm local}}\label{sec local}
\setcounter{section}{3}\setcounter{equation}{0}

Denote $\theta= \theta^+-\theta^-$ and we rewrite the system \eqref{eq 2} as follows
\begin{equation}\label{eq 3.1}
\begin{cases}
  \partial_t \theta^\pm + \partial_1( u^\pm_1 \, \theta^\pm ) + \kappa |D|^\alpha \theta^\pm =0, \quad \alpha\in ]0,2],
  \, \kappa\geq 0, \\
  u^\pm_1= \pm  \mathcal{R}_1 \mathcal{R}_2^2 |D|^{-1} \theta,\quad \theta^\pm|_{t=0} = \theta^\pm_0,
\end{cases}
\end{equation}
where the equation of $\theta^\pm$ should be understood as two equations of $\theta^+$ and $\theta^-$ respectively.

\subsection{A priori estimates}\label{subsec Apri}
In this subsection, we \textit{a priori} suppose that $\theta^\pm\in C(\mathbb{R}^+; H^m\cap L^p)$ and
$\theta\in C(\mathbb{R}^+; H^m\cap L^p)$ with $m>2$, $p\in ]1,2[$ are independent functions
and they satisfy the equation \eqref{eq 3.1}. %Note that we do not call for that $\theta =\theta^+ -\theta^-$.

We first obtain the $L^q$ estimate of $\theta^\pm$ with $q\in [p,\infty[$.
Let $\chi $ be the cut-off function introduced in the subsection \ref{subsec Pre1} and $\chi_R(\cdot)\triangleq\chi(\frac{\cdot}{R})$ for $R>0$.
Multiplying the equations of $\theta^\pm$ by $|\theta^\pm|^{q-2}\theta^\pm \chi_R $
and integrating over the spatial variable, we get
\begin{equation*}
\begin{split}
  \frac{1}{q}\frac{d}{dt} \Big(\int_{\mathbb{R}^2} |\theta^\pm|^q(t,x) \chi_R(x) \mathrm{d}x\Big)
  = &  - \int_{\mathbb{R}^2} \partial_1 (u^\pm_1 \theta^\pm)(t,x)\, |\theta^\pm|^{q-2}\theta^\pm(t,x) \chi_R(x)\mathrm{d}x -\\
  & - \kappa\int_{\mathbb{R}^2} |D|^\alpha \theta^\pm(t,x)\, |\theta^\pm|^{q-2}\theta^\pm(t,x)\chi_R(x)\mathrm{d}x \\
  \triangleq & \, I^\pm(t) + II^\pm(t).
\end{split}
\end{equation*}
For $I^\pm(t)$, from the integration by parts, we have
\begin{equation*}
\begin{split}
  I^\pm(t) &  = -\int_{\mathbb{R}^2} (\partial_1 u^\pm_1) |\theta^\pm|^q \chi_R(x)\mathrm{d}x
  - \int_{\mathbb{R}^2} u^\pm_1 \partial_1\theta^\pm \, |\theta^\pm|^{q-2}\theta^\pm \,\chi_R(x)\mathrm{d}x \\
  & = -(1-1/q) \int_{\mathbb{R}^2} (\partial_1 u^\pm_1) |\theta^\pm|^q \chi_R(x)\mathrm{d}x
  + (1/q) R^{-1}\int_{\mathbb{R}^2} u^\pm_1\,|\theta^\pm|^q\, \partial_1\chi\big(\frac{x}{R}\big) \mathrm{d}x \\
  & \leq (1-1/q)\|\partial_1 u^\pm_1(t)\|_{L^\infty} \|\theta^\pm(t)\|^q_{L^q} + (qR)^{-1}\|\partial_1\chi\|_{L^\infty}
  \|u^\pm_1(t)\|_{L^\infty} \|\theta^\pm(t)\|_{L^q}^q \\
  & \lesssim \big(\|\partial_1 u^\pm_1(t)\|_{L^\infty}+ (qR)^{-1} \|\theta(t)\|_{L^p\cap H^m}\big)
  \|\theta^\pm(t)\|_{L^q}^q,
\end{split}
\end{equation*}
where in the last line we have used the following estimation %(by Bernstein's inequality and Hardy-Littlewood-Sobolev's inequality)
\begin{equation}\label{eq uLinf1}
\begin{split}
  \|u^\pm_1(t)\|_{L^\infty_x} & \leq \|\mathcal{R}_1\mathcal{R}_2^2|D|^{-1}\Delta_{-1}\theta(t)\|_{L^\infty_x}+
  \sum_{j\geq 0}\|\mathcal{R}_1\mathcal{R}_2^2|D|^{-1}\Delta_j\theta(t)\|_{L^\infty_x} \\
  & \lesssim \||D|^{-1}\theta(t)\|_{L^{2p/(2-p)}} + \sum_{j\geq 0} 2^{-j m} (2^{j m}\|\Delta_j \theta(t)\|_{L^2}) \\
  & \lesssim \|\theta(t)\|_{L^p\cap H^m}.
\end{split}
\end{equation}
For $II^\pm(t)$, by virtue of the following pointwise inequality (cf. \cite[Proposition 3.3]{Ju})
\begin{equation*}
  |f(x)|^\beta f(x)\, \big(|D|^\alpha f\big)(x) \geq \frac{1}{\beta+2} \big(|D|^\alpha |f|^{\beta+2}\big) (x), \quad
  \forall \alpha\in [0,2],\, \beta\in [-1,\infty[,
\end{equation*}
we have
\begin{equation*}
\begin{split}
  II^\pm(t) & \leq - \frac{\kappa}{ q} \int_{\mathbb{R}^2} \big(|D|^\alpha |\theta^\pm|^q \big)(t,x)\, \chi_R(x)\mathrm{d}x \\
  & \leq -\frac{\kappa}{ q} \int_{\mathbb{R}^2} |\theta^\pm|^q (t,x)\, \big(|D|^\alpha \chi_R\big)(x)\mathrm{d}x \\
  & \leq \frac{\kappa}{ q} R^{-\alpha}  \||D|^\alpha\chi\|_{L^\infty} \|\theta^\pm(t)\|_{L^q}^q.
\end{split}
\end{equation*}
Integrating in time and gathering the upper results, and from the support property of $\chi$, we get
\begin{equation*}
\begin{split}
  \int_{|x|\leq R}|\theta^\pm(t,x)|^q \mathrm{d}x \leq\, & \int_{\mathbb{R}^2}|\theta^\pm(t,x)|^q \chi_R(x)\mathrm{d}x \\
  \leq\, & \|\theta^\pm_0\|_{L^q}^q + q C \int_0^t \|\partial_1 u^\pm_1(\tau)\|_{L^\infty}
   \|\theta^\pm(\tau)\|_{L^q}^q \mathrm{d}\tau \\  & + C \int_0^t
  \big(R^{-1}\|\theta^\pm(\tau)\|_{L^p\cap H^m} + \kappa R^{-\alpha} \big) \|\theta^\pm(\tau)\|_{L^q}^q \mathrm{d}\tau.
\end{split}
\end{equation*}
According to the monotone convergence theorem and $\theta^\pm\in C(\mathbb{R}^+; H^m\cap L^p)$,
and by passing $R$ to $\infty$, we have that for every $t\in\mathbb{R}^+$ and $q\in [p,\infty[$,
\begin{equation}\label{eq Apri1}
  \int_{\mathbb{R}^2}|\theta^\pm(t,x)|^q \mathrm{d}x
  \leq\,  \|\theta^\pm_0\|_{L^q}^q + q C \int_0^t \|\partial_1 u^\pm_1(\tau)\|_{L^\infty}
  \|\theta^\pm(\tau)\|_{L^q}^q \mathrm{d}\tau \triangleq F^\pm(t)^q.
\end{equation}
Since
\begin{equation*}
\begin{split}
  q F^\pm(t)^{q-1}\frac{d}{dt}F^\pm(t)=\frac{d}{dt}(F^\pm(t)^q) & = q C  \|\partial_1 u^\pm_1(t)\|_{L^\infty}
  \|\theta^\pm(t)\|_{L^q}^q \\
  & \leq\, q C  \|\partial_1 u^\pm_1(t)\|_{L^\infty} \|\theta^\pm(t)\|_{L^q} F^\pm(t)^{q-1},
\end{split}
\end{equation*}
we have
\begin{equation*}
  F^\pm(t)\leq F^\pm(0) + C\int_0^t \|\partial_1 u^\pm_1(\tau)\|_{L^\infty}\|\theta^\pm(\tau)\|_{L^q}\mathrm{d}\tau.
\end{equation*}
This implies that for every $t\in\mathbb{R}^+$ and $q\in [p,\infty[$,
\begin{equation}\label{eq Apri2}
  \|\theta^\pm(t)\|_{L^q}\leq \|\theta^\pm_0\|_{L^q} + C\int_0^t \|\partial_1
  u^\pm_1(\tau)\|_{L^\infty}\|\theta^\pm(\tau)\|_{L^q}\mathrm{d}\tau,
\end{equation}
where $C$ is independent of $q$.

Next we consider the $H^m$ estimate of $\theta^\pm$ with $m>2$. For every $j\in\mathbb{N}$, we apply the dyadic operator $\Delta_j$ to
the equations of $\theta^\pm$ in \eqref{eq 3.1} to get
\begin{equation*}
  \partial_t \Delta_j\theta^\pm  + \kappa |D|^\alpha \Delta_j \theta^\pm =- \Delta_j\partial_1\big(u^\pm_1\, \theta^\pm\big).
\end{equation*}
Multiplying both sides of the upper equations by $\Delta_j\theta^\pm$ and integrating over the spatial variable,
we obtain
\begin{equation*}
  \frac{1}{2}\frac{d}{dt}\|\Delta_j\theta^\pm(t)\|_{L^2}^2 + \kappa \||D|^\frac{\alpha}{2}\Delta_j\theta^\pm(t)\|_{L^2}^2
  = \int_{\mathbb{R}^2}\Delta_j \partial_1 \big(u^\pm_1\,\theta^\pm\big)(t,x)\,\Delta_j \theta^\pm(t,x) \mathrm{d}x.
\end{equation*}
Integrating on the time variable over $[0,t]$ leads to
\begin{equation*}
  \|\Delta_j\theta^\pm(t)\|_{L^2}^2 + 2\kappa \||D|^\frac{\alpha}{2}\Delta_j\theta^\pm\|_{L^2_t L^2}^2
  \leq \|\Delta_j\theta^\pm_0\|_{L^2}^2 +
  2\int_0^t \Big| \int_{\mathbb{R}^2}\Delta_j\partial_1 \big(u^\pm_1\,\theta^\pm\big)(\tau,x)\,\Delta_j \theta^\pm(\tau,x)
  \mathrm{d}x\Big| \mathrm{d}\tau.
\end{equation*}
Then, by multiplying both sides of the above equations by $2^{2jm}$ and summing over $j\in \mathbb{N}$, and from Lemma \ref{lem ExiKey}, we find
\begin{equation}\label{eq Apri3}
\begin{split}
  & \sum_{j\in\mathbb{N}}2^{2jm}\|\Delta_j\theta^\pm(t)\|_{L^2}^2
  + 2\kappa \sum_{j\in\mathbb{N}} 2^{2jm}\||D|^\frac{\alpha}{2}\Delta_j\theta^\pm\|_{L^2_t L^2}^2\leq  \\
  \leq\, \sum_{j\in\mathbb{N}}2^{2jm}& \|\Delta_j\theta^\pm_0\|_{L^2}^2
  + C\int_0^t \Big(\|\nabla u^\pm_1\|_{L^\infty}\|\theta^\pm\|_{H^m}^2
  + \|\theta^\pm\|_{L^\infty} \|\nabla u^\pm_1\|_{H^m}\|\theta^\pm\|_{H^m}\Big)(\tau) \mathrm{d}\tau.
\end{split}
\end{equation}
For $j=-1$, from \eqref{eq Apri1} and Bernstein's inequality, we directly have
\begin{equation*}
  \|\Delta_{-1}\theta^\pm(t)\|_{L^2}^2\leq \|\theta^\pm(t)\|_{L^2}^2
  \leq\,  \|\theta^\pm_0\|_{L^2}^2 + C \int_0^t \|\partial_1 u^\pm_1(\tau)\|_{L^\infty}
  \|\theta^\pm(\tau)\|_{L^2}^2 \mathrm{d}\tau.
\end{equation*}
Gathering the upper two estimates, and from $\|\cdot\|_{B^m_{2,2}}\approx \|\cdot\|_{H^m}$, we get
\begin{equation*}
  \|\theta^\pm(t)\|_{H^m}^2 \leq C_0 \|\theta^\pm_0\|_{H^m}^2 +  C\int_0^t \Big(\|\nabla u^\pm_1\|_{L^\infty}\|\theta^\pm\|_{H^m}^2
  + \|\theta^\pm\|_{L^\infty} \|\nabla u^\pm_1\|_{H^m}\|\theta^\pm\|_{H^m}\Big)(\tau) \mathrm{d}\tau.
\end{equation*}
In a similar way as obtaining \eqref{eq Apri2} from \eqref{eq Apri1}, we see that
\begin{equation}\label{eq Apri3.1}
  \|\theta^\pm(t)\|_{H^m} \leq C_0 \|\theta^\pm_0\|_{H^m} +  C\int_0^t \Big(\|\nabla u^\pm_1\|_{L^\infty}\|\theta^\pm\|_{H^m}
  + \|\theta^\pm\|_{L^\infty} \|\nabla u^\pm_1\|_{H^m}\Big)(\tau) \mathrm{d}\tau,
\end{equation}
with $C_0\geq 1$. From the Sobolev embedding and Calder\'on-Zygmund theorem, we further deduce
\begin{equation}\label{eq Apri3.3}
  \|\theta^\pm(t)\|_{H^m} \leq C_0 \|\theta^\pm_0\|_{H^m} +
  C \int_0^t \|\theta(\tau)\|_{H^m}\|\theta^\pm(\tau)\|_{H^m} \mathrm{d}\tau.
\end{equation}
Gronwall's inequality ensures that
\begin{equation*}
  \|\theta^\pm(t)\|_{H^m} \leq C_0 \|\theta^\pm_0\|_{H^m}
  e^{C \int_0^t \|\theta(\tau)\|_{H^m}\mathrm{d}\tau}.
\end{equation*}

Now, by combining \eqref{eq Apri2} with \eqref{eq Apri3.1}, we have
\begin{equation}\label{eq Apri4}
  \|\theta^\pm(t)\|_{H^m\cap L^p} \leq C_0 \|\theta^\pm_0\|_{H^m\cap L^p}
  +  C\int_0^t \Big(\|\nabla u^\pm_1\|_{L^\infty}\|\theta^\pm\|_{H^m\cap L^p}
  + \|\theta^\pm\|_{L^\infty} \|\nabla u^\pm_1\|_{H^m}\Big)(\tau) \mathrm{d}\tau.
\end{equation}
This estimate also yields
\begin{equation}\label{eq Apri5}
\begin{split}
  \|\theta^+(t)\|_{H^m\cap L^p}+ \|\theta^-(t)\|_{H^m\cap L^p} \leq&
  C_0 ( \|\theta^+_0\|_{H^m\cap L^p} +\|\theta^-_0\|_{H^m\cap L^p}) +\\
  & + C_1 \int_0^t \|\theta(\tau)\|_{H^m}\big(\|\theta^+\|_{H^m\cap L^p}+\|\theta^-\|_{H^m\cap L^p}\big)(\tau) \mathrm{d}\tau.
\end{split}
\end{equation}
Hence, for every $T>0$ satisfying that
\begin{equation}\label{eq T}
  T \leq \frac{1}{4C_0C_1(\|\theta^+_0\|_{H^m \cap L^p}+ \|\theta^-_0\|_{H^m\cap L^p})},
\end{equation}
and $\theta$ satisfying that
\begin{equation}\label{eq aprCodi}
  \|\theta\|_{L^\infty_T H^m} \leq 2 C_0 (\|\theta^+_0\|_{H^m \cap L^p}+ \|\theta^-_0\|_{H^m\cap L^p}),
\end{equation}
we have
\begin{equation}\label{eq aprCocl}
  \|\theta^+\|_{L^\infty_T(H^m\cap L^p)}+ \|\theta^-\|_{L^\infty_T(H^m\cap L^p)} \leq 2 C_0 ( \|\theta^+_0\|_{H^m\cap L^p} +\|\theta^-_0\|_{H^m\cap L^p}).
\end{equation}
From \eqref{eq Apri2} and \eqref{eq Apri3}, we moreover obtain
\begin{equation*}
  \kappa\|\theta^+\|_{L^2_T H^{m+\frac{\alpha}{2}}}+ \kappa\|\theta^-\|_{L^2_T H^{m+\frac{\alpha}{2}}}
  \lesssim_{T,\|\theta^\pm_0\|_{H^m\cap L^p}} 1.
\end{equation*}

\subsection{Uniqueness}
Assume that $(\theta^{1,+},\theta^{1,-})$ and $(\theta^{2,+},\theta^{2,-})$ belonging to $C([0,T]; H^m\cap L^p)$ ($m>2$, $p\in]1,2[$) are two solutions to the system \eqref{eq 2} with initial data $(\theta^{1,+}_0,\theta^{1,-}_0)$ and $(\theta^{2,+}_0,\theta^{2,-}_0)$ respectively. Denote
$\delta\theta^\pm\triangleq \theta^{1,\pm}-\theta^{2,\pm}$, $\delta\theta^\pm_0\triangleq \theta^{1,\pm}_0-\theta^{2,\pm}_0$,
$\theta^i\triangleq \theta^{i,+} -\theta^{i,-}$, $u^{i,\pm}_1\triangleq\pm \mathcal{R}_1\mathcal{R}_2^2|D|^{-1}\theta^i$ for $i=1,2$ and $\delta\theta\triangleq \theta^1-\theta^2= \delta\theta^+-\delta\theta^-$,
$\delta u^\pm_1\triangleq u^{1,\pm}_1-u^{2,\pm}_2=\pm \mathcal{R}_1\mathcal{R}_2^2|D|^{-1}\delta\theta$.
Then we write the equations of $\delta\theta^\pm$ as follows
\begin{equation*}
\begin{split}
  \partial_t \delta\theta^\pm + \partial_1(u^{2,\pm}_1\,\delta\theta^\pm) + \kappa |D|^\alpha \delta\theta^\pm  & =
  -\partial_1(\delta u^\pm_1\, \theta^{1,\pm}) \\
  \delta\theta^\pm |_{t=0} & = \delta\theta^\pm_0.
\end{split}
\end{equation*}
For $R>0$, let $\chi_R$ be the cut-off function introduced in the subsection \ref{subsec Apri}, then we multiply both sides of
the upper equations by $|\delta\theta^\pm|^{p-2}\delta\theta^\pm \chi_R$ and integrate on the spatial variable to obtain
\begin{equation*}
\begin{split}
  \frac{1}{p} \frac{d}{dt}\Big(\int_{\mathbb{R}^2}|\delta\theta^\pm(t,x)|^p \chi_R(x)\mathrm{d}x\Big)= &
  -\int_{\mathbb{R}^2} \partial_1\big(u^{2,\pm}_1\,\delta\theta^\pm \big)(t,x)\,
  |\delta\theta^\pm|^{p-2}\delta\theta^\pm(t,x) \chi_R(x)\mathrm{d}x - \\
  & - \kappa\int_{\mathbb{R}^2} |D|^\alpha\delta\theta^\pm(t,x) \,
  |\delta\theta^\pm|^{p-2}\delta\theta^\pm(t,x) \chi_R(x)\mathrm{d}x - \\
  & -\int_{\mathbb{R}^2} \partial_1\big(\delta u^\pm_1\,\theta^{1,\pm} \big)(t,x)\,
  |\delta\theta^\pm|^{p-2}\delta\theta^\pm(t,x) \chi_R(x)\mathrm{d}x \\
  \triangleq & A^\pm_1(t) + A^\pm_2(t) + A^\pm_3(t).
\end{split}
\end{equation*}
Similarly as estimating $I^\pm(t)$ and $II^\pm(t)$ in the subsection \ref{subsec Apri}, we get
\begin{equation*}
\begin{split}
  A^\pm_1(t)& \leq C\big(\|\partial_1 u^{2,\pm}_1(t)\|_{L^\infty} + R^{-1}\|u^{2,\pm}_1(t)\|_{L^\infty} \big)
  \|\delta\theta^\pm(t)\|_{L^p}^p \\
  & \leq C \big(\|\theta^2(t)\|_{H^m}+R^{-1}\|\theta^2(t)\|_{H^m\cap L^p}\big)\|\delta\theta^\pm(t)\|_{L^p}^p ,
\end{split}
\end{equation*}
and
\begin{equation*}
  A^\pm_2(t)\leq C\kappa R^{-\alpha} \|\delta\theta^\pm(t)\|_{L^p}^p.
\end{equation*}
For $A^\pm_3(t)$, by virtue of the H\"older inequality, Calder\'on-Zygmund theorem and
Hardy-Littlewood-Sobolev inequality, we find
\begin{equation*}
\begin{split}
  A^\pm_3(t) & = -\int_{\mathbb{R}^2} \big((\partial_1\delta u^\pm_1)\,\theta^{1,\pm}
  + \delta u^\pm \, \partial_1\theta^{1,\pm}\big)(t,x)\,
  |\delta\theta^\pm|^{p-2}\delta\theta^\pm(t,x) \chi_R(x)\mathrm{d}x \\
  & \leq \Big(\|\partial_1\delta u^\pm_1\|_{L^p}\|\theta^{1,\pm}\|_{L^\infty}+
  \|\delta u^\pm_1\|_{L^{\frac{2p}{2-p}}} \|\partial_1\theta^{1,\pm}\|_{L^2}\Big)
  \|\delta\theta^\pm\|_{L^p}^{p-1}\|\chi_R\|_{L^\infty} \\
  & \leq C \|\theta^{1,\pm}(t)\|_{H^m} \|\delta \theta(t)\|_{L^p} \|\delta\theta^\pm(t)\|_{L^p}^{p-1}.
\end{split}
\end{equation*}
Collecting the above estimates, and in a similar way as obtaining \eqref{eq Apri1}, we infer that
\begin{equation*}
  \|\delta\theta^\pm(t)\|_{L^p}^p\leq \|\delta\theta^\pm_0\|_{L^p}^p + p C \int_0^t
  \Big(\|\theta^2\|_{H^m} \|\delta\theta^\pm\|_{L^p} +
   \|\theta^{1,\pm}\|_{H^m} \|\delta \theta\|_{L^p}\Big)(\tau) \|\delta\theta^\pm(\tau)\|_{L^p}^{p-1}
   \mathrm{d}\tau.
\end{equation*}
This estimate implies that
\begin{equation}\label{eq UniEst1}
  \|\delta\theta^\pm(t)\|_{L^p}\leq \|\delta\theta^\pm_0\|_{L^p} + C \int_0^t
  \Big(\|\theta^2(\tau)\|_{H^m} \|\delta\theta^\pm(\tau)\|_{L^p} +
   \|\theta^{1,\pm}(\tau)\|_{H^m} \|\delta \theta(\tau)\|_{L^p}\Big)\mathrm{d}\tau.
\end{equation}
Hence, summing over the upper estimates of $\delta\theta^+$ and $\delta\theta^-$, we have
\begin{equation*}
  \|\delta\theta^+(t)\|_{L^p} + \|\delta\theta^-(t)\|_{L^p}\leq  \|\delta\theta^+_0\|_{L^p} + \|\delta\theta^-_0\|_{L^p}
  + \int_0^t C(\tau) \big(\|\delta\theta^+(\tau)\|_{L^p} + \|\delta\theta^-(\tau)\|_{L^p}\big)\mathrm{d}\tau,
\end{equation*}
where $C(\tau)= C \|\theta^2(\tau)\|_{H^m} + C \|\theta^{1,+}(\tau)\|_{H^m} + C \|\theta^{1,-}(\tau)\|_{H^m}$.
%linearly depends on $\|\theta^{i,\pm}(\tau)\|_{H^m}$ ($i=1,2$).
Gronwall's inequality yields that for every $t\in[0,T]$
\begin{equation*}
   \|\delta\theta^+(t)\|_{L^p} + \|\delta\theta^-(t)\|_{L^p}\leq \big( \|\delta\theta^+_0\|_{L^p} + \|\delta\theta^-_0\|_{L^p}\big)
   e^{T\|C(t)\|_{L^\infty_T}},
\end{equation*}
and this clearly guarantees the uniqueness.

\subsection{Existence}
We construct the sequences of approximate solutions $\{(\theta^{n,+},\theta^{n,-})\}_{n\in\mathbb{N}}$ as follows.
Denote $\theta^{0,\pm}(t,x)= e^{-\kappa t |D|^\alpha}\theta^\pm_0(x)$, and for each $n\in\mathbb{N}$,
$(\theta^{n+1,+},\theta^{n+1,-})$ solves the following system
\begin{equation}\label{eq apprx}
\begin{cases}
  \partial_t \theta^{n+1,\pm} + \partial_1 (u^{n,\pm}_1\, \theta^{n+1,\pm}) + \kappa |D|^\alpha\theta^{n+1,\pm}=0, \\
  u^{n,\pm}_1= \pm \mathcal{R}_1\mathcal{R}_2^2 |D|^{-1}(\theta^{n,+}-\theta^{n,-}), \\
  \theta^{n+1, \pm}|_{t=0}=\theta^\pm_0.
\end{cases}
\end{equation}
Since $\theta^\pm_0\in H^m\cap L^p$ with $m>2$, $p\in ]1,2[$, we know that
$\theta^{0,\pm}\in C(\mathbb{R}^+;H^m\cap L^p)$.
Now assuming that for each $n\in\mathbb{N}$, $\theta^{n,\pm}\in C(\mathbb{R}^+;H^m\cap L^p)$,
we further show that $\theta^{n+1,\pm}\in C(\mathbb{R}^+; H^m\cap L^p)$.
By a classical process, it is not hard to show that
$\theta^{n+1,\pm}\in C(\mathbb{R}^+; H^m)$. To prove that $\theta^{n+1,\pm}\in C(\mathbb{R}^+; L^p)$,
we use the Duhamel's formula
\begin{equation*}
  \theta^{n+1,\pm}(t,x)=e^{-\kappa t |D|^\alpha} \theta^\pm_0(x) + \int_0^t e^{-\kappa(t-\tau) |D|^\alpha}
  f^{n+1,\pm}(\tau,x)\mathrm{d}\tau
\end{equation*}
with $f^{n+1,\pm}= \partial_1 (u^{n,\pm}_1\, \theta^{n+1,\pm})$.
By a direct computation, we deduce that for every $t\in [0,\infty[$,
\begin{equation}\label{eq f1pm}
\begin{split}
  \|f^{n+1,\pm}\|_{L^\infty_t L^p} & \leq \|\partial_1 u^{n,\pm}\,\theta^{n+1,\pm}\|_{L^\infty_t L^p}+
  \|u^{n,\pm}\,\partial_1 \theta^{n+1,\pm}\|_{L^\infty_t L^p} \\
  & \leq \|\partial_1 u^{n,\pm}\|_{L^\infty_t L^p} \|\theta^{n+1,\pm} \|_{L^\infty_t L^\infty}
  + \|u^{n,\pm}\|_{L^\infty_t L^{\frac{2p}{2-p}}} \|\partial_1\theta^{n+1,\pm}\|_{L^\infty_t L^2} \\
  & \lesssim \big(\|\theta^{n,+}\|_{L^\infty_t L^p}+\|\theta^{n,-}\|_{L^\infty_t L^p}\big) \|\theta^{n+1,\pm}\|_{L^\infty_t H^m},
\end{split}
\end{equation}
thus
\begin{equation*}
\begin{split}
  \|\theta^{n+1,\pm}(t)\|_{L^p} \lesssim \|\theta^\pm_0\|_{L^p} + t \big(\| \theta^{n,+}\|_{L^\infty_t L^p}
  +\| \theta^{n,-}\|_{L^\infty_t L^p} \big) \|\theta^{n+1,\pm}\|_{L^\infty_t H^m},
\end{split}
\end{equation*}
and this implies that $\theta^{n+1,\pm}\in L^\infty(\mathbb{R}^+; L^p)$. When $\kappa=0$, in a similar manner we can show that
$\theta^{n+1,\pm}\in C(\mathbb{R}^+; L^p)$. When $\kappa>0$, for every $t,s\in [0,\infty[$, $t>s$, we have
\begin{equation*}
\begin{split}
  \theta^{n+1,\pm} (t,x)-\theta^{n+1,\pm}(s,x) = & \big(e^{- \kappa t|D|^\alpha}-e^{-\kappa s|D|^\alpha}\big)\theta^\pm_0(x) +
  \int_s^t e^{-\kappa(t-\tau)|D|^\alpha}f^{n+1,\pm}(\tau,x)\mathrm{d}\tau \\
  & + \int_0^s \big(e^{-\kappa (t-\tau)|D|^\alpha}-e^{-\kappa (s-\tau)|D|^\alpha} \big)f^{n+1,\pm}(\tau,x)\mathrm{d}\tau \\
  \triangleq \, & B_1(t,s,x) + B_2(t,s,x) +B_3(t,s,x),
\end{split}
\end{equation*}
It is obvious that
$$ \lim_{t\rightarrow s}\big( \|B_1(t,s,x)\|_{L^p_x}+\|B_2(t,s,x)\|_{L^p_x} \big) =0.$$
For $B_3$, by Bernstein's inequality, Fubini's theorem, Young's inequality and the following estimate (cf. \cite[Proposition 2.2]{HmdK}) that
\begin{equation*}
  \|e^{-h |D|^\alpha}\Delta_j f\|_{L^p}\leq C e^{-c h 2^{j\alpha}}\|\Delta_j f\|_{L^p}, \qquad \forall j\in\mathbb{N},\,p\in[1,\infty],\,h>0,
\end{equation*}
we find that
\begin{equation*}
\begin{split}
  \|B_3(t,s,x)\|_{L^p_x} \leq & \kappa \int_0^s \int_{s-\tau}^{t-\tau} \| e^{-\kappa \tau'|D|^\alpha}
  |D|^\alpha f^{n+1,\pm}(\tau,x)\|_{L^p_x}\mathrm{d}\tau'\mathrm{d}\tau \\
  \leq & \kappa\int_0^s \int_{s-\tau}^{t-\tau} \| e^{-\kappa \tau'|D|^\alpha}
  |D|^\alpha \Delta_{-1}f^{n+1,\pm}(\tau,x)\|_{L^p_x}\mathrm{d}\tau'\mathrm{d}\tau + \\
  & + \kappa\int_0^s \int_{s-\tau}^{t-\tau} \Big(\sum_{j\in\mathbb{N}}\| e^{-\kappa \tau'|D|^\alpha}
  |D|^\alpha \Delta_j f^{n+1,\pm}(\tau,x)\|_{L^p_x}\Big)\mathrm{d}\tau'\mathrm{d}\tau \\
  \lesssim & \kappa s (t-s) \|f^{n+1,\pm}\|_{L^\infty_s L^p_x} +  \kappa (t-s) \sum_{j\in\mathbb{N}}
  \int_0^s e^{-c (s-\tau) 2^{j\alpha}}2^{j\alpha}\|\Delta_j f^{n+1,\pm}(\tau)\|_{L^p_x}\mathrm{d}\tau \\
  \lesssim & \kappa s(t-s) \|f^{n+1,\pm}\|_{L^\infty_s L^p_x} + \kappa (t-s)\sum_{j\in\mathbb{N}} \|\Delta_j f^{n+1,\pm}\|_{L^1_s L^p_x}.
\end{split}
\end{equation*}
Combining the upper estimate with \eqref{eq f1pm} and \eqref{eq exiKey2} yields
\begin{equation*}
  \|B_3(t,s,x)\|_{L^p_x}\leq C \kappa s(t-s)
\end{equation*}
with $C$ depending only on $\|\theta^{n,\pm}\|_{L^\infty_s(H^m\cap L^p)}$ and $\|\theta^{n+1,\pm}\|_{L^\infty_s H^m}$.
Hence $\theta^{n+1,\pm}\in C(\mathbb{R}^+; H^m\cap L^p)$. By induction,
we have $\theta^{n,\pm}\in C(\mathbb{R}^+; H^m\cap L^p)$ for every $n\in\mathbb{N}$.

We also show that $\{(\theta^{n,+},\theta^{n,-})\}_{n\in\mathbb{N}}$ are $n$-uniformly bounded in $C([0,T]; H^m\cap L^p)$
with $T$ defined by \eqref{eq T}, that is,
\begin{equation}\label{eq UnifBd}
  \|\theta^{n,+}\|_{L^\infty_T(H^m\cap L^p)}+ \|\theta^{n,-}\|_{L^\infty_T(H^m\cap L^p)} \leq 2 C_0 ( \|\theta^+_0\|_{H^m\cap L^p} +\|\theta^-_0\|_{H^m\cap L^p}).
\end{equation}
Indeed, from \eqref{eq T}-\eqref{eq aprCocl}, it reduces to prove that \eqref{eq aprCodi} is satisfied for every $n\in\mathbb{N}$.
This can be seen from the estimate that $\|\theta^{0,+}-\theta^{0,-}\|_{L^\infty_T H^m}\leq \|\theta^+_0-\theta^-_0\|_{H^m}\leq
2C_0\big(\|\theta^+_0\|_{H^m\cap L^p}+ \|\theta^-_0\|_{H^m\cap L^p} \big)$ and the induction method.

Next we show that $\{\theta^{n,\pm}\}_{n\in\mathbb{N}}$ are convergent in $C([0,T']; L^p)$ with some $T'\in ]0,T]$
fixed later. For $n,k \in \mathbb{N}$, $n>k$, denote
$\theta^{n,k,\pm}\triangleq \theta^{n+1,\pm}-\theta^{k+1,\pm}$, and the difference equations write
\begin{equation*}
\begin{cases}
  \partial_t\theta^{n,k,\pm} + \partial_1(u^{n+1,\pm}_1\,\theta^{n,k,\pm}) + \kappa |D|^\alpha \theta^{n,k,\pm}   =
  -\partial_1( u^{n,k,\pm}_1\, \theta^{k+1,\pm}) \\
  u^{n,k,\pm}_1\triangleq \pm \mathcal{R}_1\mathcal{R}_2^2|D|^{-1}(\theta^{n-1,k-1,+}-\theta^{n-1,k-1,-}), \\
  \theta^{n,k,\pm} |_{t=0} = 0.
\end{cases}
\end{equation*}
In a similar way as obtaining \eqref{eq UniEst1}, we get
\begin{equation*}
\begin{split}
  \|\theta^{n,k,\pm}(t)\|_{L^p}\leq &  C \int_0^t  \|(\theta^{n+1,+}-\theta^{n+1,-})(\tau)\|_{H^m}
  \|\theta^{n,k,\pm}(\tau)\|_{L^p}\mathrm{d}\tau \, + \\
  &  \,+ C\int_0^t \|\theta^{k+1,\pm}(\tau)\|_{H^m} \|(\theta^{n-1,k-1,+}-\theta^{n-1,k-1,-})(\tau)\|_{L^p} \mathrm{d}\tau.
\end{split}
\end{equation*}
Denoting $\Theta^{n,k}(t)\triangleq \|\theta^{n,k,+}(t)\|_{L^p} + \|\theta^{n,k,-}(t)\|_{L^p}$ for every $t\in [0,T]$,
we further have
\begin{equation*}
  \Theta^{n,k}(t)\leq \int_0^t h_n(\tau) \Theta^{n,k}(\tau)\mathrm{d}\tau
  + \int_0^t h_k(\tau) \Theta^{n-1,k-1}(\tau)\mathrm{d}\tau
\end{equation*}
where $h_i(\tau)=C\|\theta^{i+1,+}(\tau)\|_{H^m} + C\|\theta^{i+1,-}(\tau)\|_{H^m}$, $i=n,k$
satisfies the uniform estimate $\|h_i(\tau)\|_{L^\infty_T}\leq C M$ with $M$ an upper bound from \eqref{eq UnifBd}.
Hence, Gronwall's inequality leads to that for every $t\in [0,T]$
\begin{equation*}
\begin{split}
  \Theta^{n,k}(t) & \leq e^{\int_0^t h_n(\tau)\mathrm{d}\tau} \int_0^t h_k(\tau) \Theta^{n-1,k-1}(\tau)\mathrm{d}\tau \\
  & \leq e^{CMt }t C M \Theta^{n-1,k-1}(t).
\end{split}
\end{equation*}
By choosing $t$ small enough, i.e., for $t\in [0,T']$ (noting that $T'$ still only depends on $\|\theta^\pm_0\|_{H^m\cap L^p}$),
then there exists a constant $\mu<1$ such that
\begin{equation*}
  \Theta^{n,k}(t)\leq \mu \Theta^{n-1,k-1}(t), \quad \forall t\in[0,T'].
\end{equation*}
From iteration, we find that for every $n,k\in\mathbb{N}$, $n>k$,
\begin{equation*}
\begin{split}
  \Theta^{n,k}(t) & \leq \mu^{k+1} \big(\|\theta^{n-k,+}+\theta^{0,+}\|_{L^\infty_t L^p}
  +\|\theta^{n-k,-}+\theta^{0,-}\|_{L^\infty_t L^p}\big) \\
  & \leq CM \mu^{k+1}.
\end{split}
\end{equation*}
This ensures that $\{\theta^{n,\pm}\}_{n\in \mathbb{N}}$ are Cauchy sequences in $C([0,T']; L^p)$.
Therefore there exist $\theta^\pm\in C([0,T']; L^p)$ such that $\theta^{n,\pm}\rightarrow \theta^\pm$ strongly in
$C([0,T']; L^p)$.

Now we consider more properties of the limiting functions $\theta^\pm$. From \eqref{eq UnifBd} and
interpolation, we have that for every $\tilde{m}\in [0,m[$,
\begin{equation*}
\begin{split}
  \|\theta^{n,\pm}-\theta^\pm\|_{L^\infty_{T'} H^{\tilde{m}}(\mathbb{R}^2)} & \lesssim
  \|\theta^{n,\pm}-\theta^\pm\|_{L^\infty_{T'}L^p(\mathbb{R}^2)}^\gamma
  \|\theta^{n,\pm}-\theta^\pm\|_{L^\infty_{T'}H^m(\mathbb{R}^2)}^{1-\gamma} \\
  & \lesssim M^{1-\gamma} \|\theta^{n,\pm}-\theta^\pm\|_{L^\infty_{T'}L^p(\mathbb{R}^2)}^\gamma,
\end{split}
\end{equation*}
where $\gamma=\frac{m-{\tilde{m}}}{m+2/p-1}$.
Hence $\theta^{n,\pm}\rightarrow \theta^\pm$ strongly in $C([0,T']; H^{\tilde{m}})$ with ${\tilde{m}}\in [0,m[$.
By a classical argument, we know that $\theta^\pm$ solve the limiting equations \eqref{eq 2}, and if $m>3$,
they satisfy the equations in the classical sense. From Fatou's lemma, we get $\theta^\pm\in L^\infty([0,T']; H^m)$.

Similarly as proving the corresponding point in Theorem 1.1 of \cite{LiRZ}, we can also show that
$\theta^\pm\in C([0,T']; H^m)\cap C^1([0,T'[;H^{m_0})$ with $m_0=\min\{m-1,m-\alpha\}$.

\subsection{Blowup Criterion}
First we know that the system \eqref{eq 2} has a natural blowup criterion: if $T^*<\infty$, then necessarily
\begin{equation*}
  \|\theta^+\|_{L^\infty([0,T^*[; H^m\cap L^p )}+ \|\theta^-\|_{L^\infty([0,T^*[;H^m\cap L^p)} =\infty.
\end{equation*}
Otherwise the solution will go beyond the time $T^*$.

Next, from \eqref{eq Apri4} and the Calder\'on-Zygmund theorem, we find
\begin{equation*}
  \|\theta^\pm(t)\|_{H^m\cap L^p} \leq C_0 \|\theta^\pm_0\|_{H^m\cap L^p}
  +  C\int_0^t \Big(\|\mathcal{R}_1\mathcal{R}_2^2|D|^{-1}\nabla \theta\|_{L^\infty}\|\theta^\pm\|_{H^m\cap L^p}
  + \|\theta^\pm\|_{L^\infty} \|\theta\|_{H^m}\Big)(\tau) \mathrm{d}\tau,
\end{equation*}
Denote $G(t)=\|\theta^+(t)\|_{H^m\cap L^p}+\|\theta^-(t)\|_{H^m\cap L^p}$ for every $t\in [0,T^*[$,
then from Lemma \ref{lem LogInq} and $\theta=\theta^+-\theta^-$, we get
\begin{equation*}
\begin{split}
  G(t)& \leq C_0 G(0) + C \int_0^t \Big(\|\mathcal{R}_1\mathcal{R}_2^2|D|^{-1}\nabla\theta(\tau)\|_{L^\infty}
  +\|\theta^+(\tau)\|_{L^\infty}+\|\theta^-(\tau)\|_{L^\infty}\Big)\, G(\tau)\mathrm{d}\tau \\
  & \leq C_0G(0)+ C\int_0^t \Big(1+\|\theta^+(\tau)\|_{L^\infty}+\|\theta^-(\tau)\|_{L^\infty}\Big)
  \log\big(e+G(\tau)\big)G(\tau)\mathrm{d}\tau.
\end{split}
\end{equation*}
Direct computation yields that for every $t\in[0,T^*[$,
\begin{equation*}
  G(t)\leq (C_0 G(0)+e)^{\exp\big\{Ct+ C\int_0^t(\|\theta^+(\tau)\|_{L^\infty}+\|\theta^-(\tau)\|_{L^\infty})\mathrm{d}\tau \big\}}.
\end{equation*}
Therefore, if $T^*<\infty$, we necessarily need that $\int_0^{T^*}(\|\theta^+(t)\|_{L^\infty}+\|\theta^-(t)\|_{L^\infty})\mathrm{d}t=\infty$.

%%%%%%%%%%%%%%%%%%%%%%%%%%%%%%%%%%%%%%%%%%%%%%%%%%%%%%%%%%%%%%%%%%%%%%%%%%%%%%%%%%%%%%%%%%%%%%%%%%%%%%%%%%%%%%%%%%%
\section{Proof of Proposition \ref{prop FurP}}
\setcounter{section}{4}\setcounter{equation}{0}

Throughout this section, we assume that $(\theta^+,\theta^-)\in C([0,T^*[; H^m\cap L^p)\cap C^1([0,T^*[;H^{m_0})$
with $m>4$, $p\in]1,2[$, $m_0=\min\{m-1,m-\alpha\}$ is the corresponding maximal lifespan solution obtained in Theorem
\ref{thm local}.

\subsection{Proof of Proposition \ref{prop FurP}-(1): the non-negativity of the solutions.}
For every $T\in ]0,T^*[$, denote $U_T=]0,T]\times \mathbb{R}^2$.
According to the Sobolev embedding, we infer that
\begin{equation*}
  \theta^\pm \in C^{1,0}_{t,x}(U_T)\cap C^{0,2}_{t,x}(U_T)\cap
  C^0_{t,x}(\overline{U}_T) \cap L^2(U_T)
\end{equation*}
satisfies
\begin{equation*}
  \sup_{U_T} \big(|\partial_t \theta^\pm| + |\nabla \theta^\pm| + |\nabla^2\theta^\pm|\big) + \sup_{\overline{U}_T} |\theta^\pm|
  \lesssim_{\|\theta^\pm\|_{L^\infty_T(H^m\cap L^p)}} 1,
\end{equation*}
and $\theta^\pm$ solve the following equations pointwise
\begin{equation*}
\begin{cases}
  \partial_t \theta^\pm + \nabla\cdot(u^\pm \, \theta^\pm) = \kappa |D|^\alpha \theta^\pm, \\
  u^\pm = \pm \big(\mathcal{R}_1 \mathcal{R}_2^2 |D|^{-1}(\theta^+-\theta^-),0 \big), \\
  \theta^\pm(0,x) =\theta^\pm_0(x).
\end{cases}
\end{equation*}
To show that $u^\pm\in C^{0,1}_{t,x}(U_T)$, noticing $\nabla u^\pm= \nabla(\mathcal{R}_1\mathcal{R}_2^2 |D|^{-1}(\theta^+-\theta^-),0)$,
it suffices to prove that $\theta^\pm \in C(]0,T]; B^0_{\infty,1})$, and this turns out to be a consequence of $\partial_t\theta^\pm\in C([0,T]; H^{m_0})$
with $m_0=\min\{m-1,m-\alpha\}$ and Sobolev's embedding. It is also clear to see that $\theta^\pm_0\in C^2(\mathbb{R}^2)$ and
\begin{equation*}
  \sup_{U_T} |\mathrm{div} u^\pm| = \sup_{U_T} |\mathcal{R}_1^2\mathcal{R}_2^2 (\theta^+-\theta^-)|
  \lesssim \|\theta^+\|_{L^\infty_T H^m}+ \|\theta^-\|_{L^\infty_T H^m}.
\end{equation*}
Hence by virtue of Lemma \ref{lem Pos}, and from $\theta^\pm_0\geq 0$, we have $\theta^\pm \geq 0$ in $U_T$.
Since $T\in ]0,T^*[$ is arbitrary, this implies $\theta^\pm\geq 0$ for all $[0,T^*[\times \mathbb{R}^2$.

\subsection{Proof of Proposition \ref{prop FurP}-(2).}\label{sec FurP2}
Let $T\in]0,T^*[$ be arbitrary, $\phi \in C^\infty(\mathbb{R})$ be an even cut-off function satisfying that
\begin{equation*}
  0\leq \phi \leq 1,\quad \mathrm{supp}\,\phi \subset ]-2,2[,\quad \phi\equiv 1\; \mathrm{on}\; [-1,1].
\end{equation*}
Denote $\phi_R(\cdot)=\phi(\frac{\cdot}{R})$ for $R>0$.
Multiplying both sides of the equations of $\theta^\pm$ by $\phi_R(x_1)$ and integrating over the $x_1$-variable, we get
\begin{equation*}
\begin{aligned}
  \frac{d}{dt}\int_{\mathbb{R}} \theta^\pm(t,x)\phi_R(x_1)\mathrm{d}x_1
  & =- \int_{\mathbb{R}} \partial_1(u^\pm_1\, \theta^\pm)(t,x)\phi_R(x_1)\mathrm{d} x_1
  -\kappa \int_{\mathbb{R}}|D|^\alpha\theta^\pm (t,x)\phi_R(x_1)\mathrm{d}x_1 \\
  & \triangleq \mathrm{I}^\pm(t,x_2) + \mathrm{II}^\pm(t,x_2),
\end{aligned}
\end{equation*}
with $u^\pm_1\triangleq \pm \mathcal{R}_1\mathcal{R}_2^2 |D|^{-1}(\theta^+-\theta^-)$.
For $\mathrm{I}^\pm$, from the integration by parts and H\"older's inequality, we obtain that for every
$(t,x_2)\in[0,T]\times \mathbb{R}$
\begin{equation*}
\begin{aligned}
  \mathrm{I}^\pm(t,x_2) &= \frac{1}{R}\int_{\mathbb{R}}u^\pm_1(t,x)\theta^\pm(t,x) (\partial_1\phi)(x_1/R)\mathrm{d}x_1 \\
  & \leq \frac{1}{R^{1/2}}\|u^\pm_1\|_{L^\infty_T L^\infty_x} \|\theta^\pm\|_{L^\infty_T L^{\infty,2}_{x_2,x_1}}
  \|\nabla\phi\|_{L^2}.
\end{aligned}
\end{equation*}
From \eqref{eq uLinf1}, we see
\begin{equation}\label{eq uLinf}
  \|u^\pm_1\|_{L^\infty_T L^\infty_x} \lesssim \|\theta^+\|_{L^\infty_T (H^m\cap L^p)} + \|\theta^-\|_{L^\infty_T (H^m\cap L^p)}.
\end{equation}
By the Sobolev embedding, we also find that
\begin{equation*}
  \|\theta^\pm\|_{L^\infty_T L^{\infty,2}_{x_2,x_1}}\lesssim \|(\mathrm{Id}+|D_2|)\theta^\pm\|_{L^\infty_T L^2_x}
  \lesssim \|\theta^\pm\|_{L^\infty_T H^m}.
\end{equation*}
Thus
\begin{equation}
  \|\mathrm{I}^\pm\|_{L^\infty_T L^\infty_{x_2}}  \lesssim \frac{1}{R^{1/2}}  (\|\theta^+\|^2_{L^\infty_T (H^m\cap L^p)}
  + \|\theta^-\|^2_{L^\infty_T (H^m\cap L^p)})\|\nabla\phi\|_{L^2}.
\end{equation}
We can rewrite $\mathrm{II}^\pm$ as follows
\begin{equation*}
\begin{split}
  \mathrm{II}^\pm(t,x_2)  & = -\kappa \int_{\mathbb R} |D_2|^\alpha \theta^\pm(t,x) \phi_R(x_1)\mathrm{d}x_1
  -\kappa\int_{\mathbb R} (|D|^\alpha -|D_2|^\alpha)\theta^\pm(t,x) \phi_R(x_1)\mathrm{d}x_1  \\
  & \triangleq \mathrm{II}^\pm_1(t,x_2)  + \mathrm{II}_2^\pm(t,x_2).
\end{split}
\end{equation*}
It is obvious to see
\begin{equation*}
  \mathrm{II}^\pm_1(t,x_2)=-\kappa |D_2|^\alpha \Big(\int_{\mathbb R} \theta^\pm(t,x) \phi_R(x_1)\mathrm{d}x_1\Big).
\end{equation*}
For $\mathrm{II}^\pm_2$, observe that
\begin{equation*}
\begin{split}
  \mathrm{II}^\pm_2(t,x_2) & =-\kappa\int_{\mathbb{R}}\Big(\frac{|D|^\alpha-|D_2|^\alpha}{|D_1|^\alpha}\theta^\pm\Big)(t,x)\,
  |D_1|^\alpha(\phi_R)(x_1)\mathrm{d}x_1 \\
  & =-\kappa R^{-\alpha}\int_{\mathbb{R}}\Big(\frac{|D|^\alpha-|D_2|^\alpha}{|D_1|^\alpha} \theta^\pm\Big)(t,x)\,
  \big(|D_1|^\alpha\phi\big)\big(\frac{x_1}{R}\big)\mathrm{d}x_1.
\end{split}
\end{equation*}
If $\alpha\in]1/2,2]$, from H\"older's inequality and the fact that $|\zeta|^\alpha-|\zeta_2|^\alpha\leq |\zeta_1|^\alpha$
for all $\alpha\in]0,2]$, $\zeta=(\zeta_1,\zeta_2)\in\mathbb{R}^2$, we obtain
\begin{equation*}
\begin{split}
  \|\mathrm{II}^\pm_2\|_{L^\infty_T L^\infty_{x_2}} & \leq \kappa R^{-\alpha} \Big\|\frac{|D|^\alpha-|D_2|^\alpha}{|D_1|^\alpha}
  \theta^\pm\Big\|_{L^\infty_T L^{\infty,2}_{\tilde{x}_2,x_1}}
  \Big\|\big(|D_1|^\alpha\phi\big)\big(\frac{x_1}{R}\big)\Big\|_{L^2_{x_1}} \\
  & \lesssim \kappa R^{-(\alpha-\frac{1}{2})} \Big\|(\mathrm{Id}+|D_2|)\frac{|D|^\alpha-|D_2|^\alpha}{|D_1|^\alpha}
  \theta^\pm \Big\|_{L^\infty_T L^2_x} \||D|^\alpha\phi\|_{L^2} \\
  & \lesssim \kappa R^{-(\alpha-\frac{1}{2})} \|\theta^\pm\|_{L^\infty_T H^m}\||D|^\alpha\phi\|_{L^2},
\end{split}
\end{equation*}
where we also have used the estimate that $\|f\|_{L^{\infty,2}_{x_2,x_1}}\leq \|f\|_{L^{2,\infty}_{x_1,x_2}} \lesssim \|(\mathrm{Id}+|D_2|)f\|_{L^2_x}$.
Since
\begin{equation*}
  \frac{d}{dt}\int_{\mathbb{R}} \theta^\pm(t,x)\phi_R(x_1)\mathrm{d}x_1
  + \kappa |D_2|^\alpha\Big( \int_{\mathbb{R}} \theta^\pm(t,x)\phi_R(x_1)\mathrm{d}x_1 \Big)
  = \mathrm{I}^\pm(t,x_2) + \mathrm{II}^\pm_2(t,x_2),
\end{equation*}
we get
\begin{equation*}
\begin{split}
  \Big\|\int_{|x_1|\leq R}\theta^\pm(t,x)\mathrm{d}x_1 \Big\|_{L^\infty_T L^\infty_{x_2}}
  & \leq \Big\|\int_{\mathbb{R}}\theta^\pm(t,x)\phi_R(x_1)\mathrm{d}x_1 \Big\|_{L^\infty_T L^\infty_{x_2}} \\
  & \leq \Big\|\int_{\mathbb{R}}\theta^\pm_0(x)\phi_R(x_1)\mathrm{d}x_1 \Big\|_{L^\infty_{x_2}}
  + T \big(\|\mathrm{I}^\pm\|_{L^\infty_T L^\infty_{x_2}} +\|\mathrm{II}_2^\pm\|_{L^\infty_T L^\infty_{x_2}}\big) \\
  & \leq \|\theta^\pm_0\|_{L^{\infty,1}_{x_2,x_1}} + C T \big(R^{-\frac{1}{2}}+ R^{-(\alpha-\frac{1}{2})}\big),
\end{split}
\end{equation*}
where $C$ is an absolute constant depending on $\kappa$, $\|\theta^\pm\|_{L^\infty_T (H^m\cap L^p)}$ and $\phi$.
From $\theta^\pm(t)\geq 0$ for all $t\in[0,T]$ and the monotone convergence theorem, and by passing $R$ to infinity, we find
\begin{equation*}
  \|\theta^\pm\|_{L^\infty_T L^{\infty,1}_{x_2,x_1}}\leq \|\theta^\pm_0\|_{L^{\infty,1}_{x_2,x_1}} .
\end{equation*}
Hence this estimate combined with the fact that $T\in ]0,T^*[$ is arbitrary leads to \eqref{eq MP}.

Now, since $\theta^\pm\in C([0,T^*[;H^m\cap L^p)$ with $m>4$ and $p\in]1,2[$, we have
\begin{equation}\label{eq FACT}
  \lim_{x_1\rightarrow -\infty}\Big(\theta^\pm(t,x)+ \sum_{k=1,2,3} |\nabla^k \theta^\pm(t,x)| \Big)=0,
  \qquad \forall (t,x_2)\in[0,T^*[ \times\mathbb{R},
\end{equation}
thus we moreover deduce that for every $t\in[0,T^*[$
\begin{equation}\label{eq rhoLinf}
  \|\rho^\pm(t,x)\|_{L^\infty_x} \leq \Big\|\int_{-\infty}^{x_1}\theta^\pm(t,\tilde{x}_1,x_2)
  \mathrm{d}\tilde{x}_1\Big\|_{L^\infty_{x}}
  \leq \Big\|\int_\mathbb{R} \theta^\pm(t,x)\mathrm{d}x_1\Big\|_{L^\infty_{x_2}}\leq \|\theta^\pm_0\|_{L^{\infty,1}_{x_2,x_1}},
\end{equation}
and
\begin{equation}\label{eq FACT1}
  \lim_{x_1\rightarrow -\infty} \rho^\pm(t,x)=\lim_{x_1\rightarrow -\infty} \int_{-\infty}^{x_1}\theta^\pm(t,\tilde{x}_1,x_2)
  \mathrm{d}\tilde{x}_1=0, \quad \forall (t,x_2)\in[0,T^*[\times\mathbb{R}.
\end{equation}
Next we shall justify that $\rho^\pm$ are the mild solutions of the system \eqref{eq 1} for $(t,x)\in [0,T^*[\times \mathbb{R}^2$.
From Theorem \ref{thm local}, we know that
\begin{equation*}
  \theta^\pm(t,x)=e^{-\kappa t|D|^\alpha}\theta^\pm_0(x) - \int_0^t e^{-\kappa(t-\tau)|D|^\alpha}
  \partial_1(u^\pm \,\theta^\pm)(\tau,x)\mathrm{d}\tau, \quad \forall (t,x)\in [0,T^*[\times \mathbb{R}^2,
\end{equation*}
with
\begin{equation*}
  u^\pm_1 =\pm \mathcal{R}_1\mathcal{R}_2^2|D|^{-1}(\theta^+-\theta^-)= \pm \mathcal{R}_1^2\mathcal{R}_2^2(\rho^+-\rho^-).
\end{equation*}
Taking advantage of the relation $\theta^\pm=\partial_1\rho^\pm$, we get
\begin{equation*}
\begin{split}
  \rho^\pm(t,x) =& \int_{-\infty}^{x_1}\theta^\pm(t,\tilde{x}_1,x_2)\mathrm{d}\tilde{x}_1 \\
  =& \int_{-\infty}^{x_1} e^{-\kappa t|D|^\alpha}\partial_1\rho^\pm_0(\tilde{x}_1,x_2) \mathrm{d}\tilde{x}_1
  - \int_{-\infty}^{x_1} \int_0^t e^{-\kappa(t-\tau)|D|^\alpha}\partial_1(u_1^\pm \partial_1\rho^\pm)(\tau,\tilde{x}_1,x_2)
  \mathrm{d}\tau\mathrm{d}\tilde{x}_1 \\
  = & \, e^{-\kappa t|D|^\alpha}\rho^\pm_0(x) -\int_0^t e^{-\kappa(t-\tau)|D|^\alpha}(u_1^\pm\, \partial_1\rho^\pm)(\tau,x)
  \mathrm{d}\tau+ E(t,x_2),
\end{split}
\end{equation*}
with
\begin{equation*}
\begin{split}
  E(t,x_2)& = E_1(t,x_2)+ E_2(t,x_2) \\
  & \triangleq   - \lim_{\tilde{x}_1\rightarrow -\infty}e^{-\kappa t|D|^\alpha}\rho^\pm_0(\tilde{x}_1,x_2)
  + \lim_{\tilde{x}_1\rightarrow -\infty} \int_0^t e^{-\kappa(t-\tau)|D|^\alpha}
  (u_1^\pm\,\theta^\pm)(\tau,\tilde{x}_1,x_2)\mathrm{d}\tau.
\end{split}
\end{equation*}
When $\kappa=0$, by virtue of \eqref{eq FACT1}, \eqref{eq uLinf} and \eqref{eq FACT}, it just reduces to
\begin{equation*}
  \rho^{\pm}(t,x)= \rho^\pm_0(x) - \int_0^t (u_1^\pm\, \partial_1\rho^\pm)(\tau,x) \mathrm{d}\tau.
\end{equation*}
%Note that at this case, $\rho^\pm$ also solve the system \eqref{eq 1} in the pointwise sense
%\begin{equation*}
%  \partial_t \rho^\pm + u_1^\pm \,\partial_1\rho^\pm =0, \qquad  \rho^\pm|_{t=0}=\rho^\pm_0.
%\end{equation*}
When $\kappa>0$, noticing that
\begin{equation}\label{eq Exp2}
  e^{-\kappa t|D|^\alpha}\rho^\pm_0(x)=\int_{\mathbb{R}^2} K_\alpha(\kappa t,y)\rho^\pm_0(x-y)\mathrm{d}y, \quad \alpha\in]0,2],
\end{equation}
where $K_\alpha(\kappa t,y)=(\kappa t)^{-2/\alpha}K_\alpha(y/(\kappa t)^{1/\alpha})$ and
$K_\alpha(y)=\mathcal{F}^{-1}(e^{-|\zeta|^\alpha})(y)$ ($\alpha\in ]0,2]$) satisfies
\begin{equation}\label{eq KerPro}
  K_\alpha\geq 0,\quad
  \begin{cases}
   K_\alpha(y)\approx \frac{1}{(1+|y|^2)^{(2+\alpha)/2}}, & \quad  y\in\mathbb{R}^2,\, \alpha\in]0,2[, \\
  K_2(y) = \frac{1}{4}e^{-|y|^2/4}, & \quad y\in \mathbb{R}^2,\, \alpha=2,
  \end{cases}
\end{equation}
thus from \eqref{eq rhoLinf}, \eqref{eq FACT1} and the dominated convergence theorem, we find $E_1(t,x_2)=0$.
Similarly, from \eqref{eq uLinf} and \eqref{eq FACT}, we also get $E_2(t,x_2)=0$. Hence we have for every
$(t,x)\in[0,T^*[\times\mathbb{R}^2$,
\begin{equation}\label{eq bbb}
  \rho^\pm(t,x)= e^{-\kappa t|D|^\alpha}\rho^\pm_0(x) -\int_0^t e^{-\kappa(t-\tau) |D|^\alpha}
  (u_1^\pm\, \partial_1\rho^\pm)(\tau,x) \mathrm{d}\tau.
\end{equation}

\subsection{Proof of Proposition \ref{prop FurP}-(3).}
We first show that for $k=1,2,3$, $\nabla^k\rho^\pm(t)\in L^\infty_x$ for all $t\in[0,T^*[$ under some appropriate assumptions of $\rho^\pm_0$.
Clearly, since $\nabla^{k-1}\partial_1\rho^\pm(t)=\nabla^{k-1}\theta^\pm(t)\in L^\infty_x$ for all $t\in [0,T^*[$,
it suffices to consider the case of $\partial^k_2\rho^\pm$. Due to that $\theta^\pm\in C([0,T^*[;H^m\cap L^p)$
with $m>4$ and $p\in ]1,2[$, the nonlinear term satisfies that for every $T\in]0,T^*[$,
\begin{equation}\label{eq FACT2}
\begin{split}
  \|\partial_2^k(u^\pm_1\,\partial_1\rho^\pm)\|_{L^\infty_T L^\infty_x} & \leq
  \sum_{0\leq j\leq k} \|\partial_2^j u^\pm_1 \, \partial_2^{k-j}\theta^\pm\|_{L^\infty_T L^\infty_x} \\
  & \leq \sum_{0\leq j\leq k} \|\partial_2^j u^\pm_1\|_{L^\infty_T L^\infty_x}\|\partial_2^{k-j}\theta^\pm\|_{L^\infty_T L^\infty_x}\\
  & \lesssim (\|\theta^+\|_{L^\infty_T (H^m\cap L^p)}+\|\theta^-\|_{L^\infty_T (H^m\cap L^p)})\|\theta^\pm\|_{L^\infty_T H^m}.
\end{split}
\end{equation}
Thus, from \eqref{eq bbb} and $\partial_2^k\rho^\pm_0\in L^\infty_x$, we have
\begin{equation}\label{eq par^kRho}
  \partial_2^k\rho^\pm(t,x)= e^{-\kappa t|D|^\alpha}\partial_2^k\rho^\pm_0(x)
  -\int_0^t e^{-\kappa(t-\tau)|D|^\alpha}\partial_2^k(u_1^\pm\, \partial_1\rho^\pm)(\tau,x) \mathrm{d}\tau,
\end{equation}
and
\begin{equation*}
  \|\partial_2^k\rho^\pm\|_{L^\infty_T L^\infty_x}\leq \|\partial_2^k \rho^\pm_0\|_{L^\infty_x}+ C T,
\end{equation*}
with $C$ depending on $\|\theta^\pm\|_{L^\infty_T (H^m\cap L^p)}$, which implies that
$\partial_2^k\rho^\pm(t)\in L^\infty_x$ for all $t\in[0,T^*[$. Moreover, thanks to $\lim_{x_1\rightarrow -\infty}\partial_2^k\rho^\pm_0(x)=0$
for every $x_2\in\mathbb{R}$, \eqref{eq FACT} and the dominated convergence theorem, we also have
\begin{equation}\label{eq lim1}
  \lim_{x_1\rightarrow -\infty}\partial_2^k\rho^\pm(t,x)=0,\qquad \forall (t,x_2)\in [0,T^*[\times \mathbb{R}.
\end{equation}

Next we show that $\rho^\pm$ solve the system \eqref{eq 1} in the classical pointwise sense.
Since $\theta^\pm$ are the classical solutions to the system \eqref{eq 2} and $\partial_1\rho^\pm =\theta^\pm$, we have
that for every $(t,x)\in]0,T^*[\times\mathbb{R}^2$,
\begin{equation*}
\begin{aligned}
  \partial_t \rho^\pm (t,x) & = \int_{-\infty}^{x_1}\partial_t \theta^\pm(t,\tilde{x}_1,x_2)\mathrm{d}\tilde{x}_1 \\
  & =-\int_{-\infty}^{x_1}\partial_1(u^\pm_1\, \theta^\pm )(t,\tilde{x}_1,x_2)\mathrm{d}\tilde{x}_1
  - \kappa \int_{-\infty}^{x_1}|D|^\alpha\theta^\pm(t,\tilde{x}_1,x_2)\mathrm{d}\tilde{x}_1 \\
  & =-\int_{-\infty}^{x_1}\partial_1(u^\pm_1\, \partial_1\rho^\pm )(t,\tilde{x}_1,x_2)\mathrm{d}\tilde{x}_1
  - \kappa \int_{-\infty}^{x_1}\partial_1|D|^\alpha \rho^\pm(t,\tilde{x}_1,x_2)\mathrm{d}\tilde{x}_1 \\
  & =-u^\pm_1\,\partial_1\rho^\pm(t,x) - \kappa |D|^\alpha \rho^\pm(t,x) + \widetilde{E}^\alpha(t,x_2),
\end{aligned}
\end{equation*}
where
\begin{equation*}
  \widetilde{E}^\alpha(t,x_2)=\lim_{x_1\rightarrow -\infty} |D|^\alpha \rho^\pm(t,x),
\end{equation*}
and in the last line we have used \eqref{eq uLinf} and \eqref{eq FACT}. When $\alpha=2$, from \eqref{eq FACT} and \eqref{eq lim1},
we directly get $\widetilde{E}^2(t,x)=0$. When $\alpha\in ]0,2[$, due to $\rho^\pm\in L^\infty([0,T^*[; C^2_b(\mathbb{R}^2))$,
from Lemma \ref{lem DalpExp} we have
\begin{equation*}
\begin{split}
  |D|^\alpha\rho^\pm(t,x)  = & -c_\alpha \bigg( \int_{B_1}\frac{\rho^\pm(t,x+y)-\rho^\pm(t,x)
  -\nabla\rho^\pm(t,x)\cdot y }{|y|^{2+\alpha}} \mathrm{d}y +
  \int_{B_1^c}\frac{\rho^\pm(t,x+y)-\rho^\pm(t,x)}{|y|^{2+\alpha}} \mathrm{d}y \bigg) \\
  = &-c_\alpha\bigg( \int_{B_1}\int_0^1\int_0^1\frac{\big(y\cdot \nabla^2\rho^\pm(t,x+ s\tau y)\big)\cdot y }
  {|y|^{2+\alpha}}\tau\mathrm{d}s\mathrm{d}\tau\mathrm{d}y
  + \int_{B_1^c}\frac{\rho^\pm(t,x+y)-\rho^\pm(t,x)}{|y|^{2+\alpha}} \mathrm{d}y \bigg),
\end{split}
\end{equation*}
and by the dominated convergence theorem, \eqref{eq FACT} and \eqref{eq lim1}, we find $\widetilde{E}^\alpha(t,x_2)=0$.

Similarly, we can prove that $\nabla\rho^\pm$ solve the equations in the classical pointwise sense
\begin{equation*}
  \partial_t (\nabla\rho^\pm) + u^\pm_1\, \partial_1 (\nabla \rho^\pm)+ \kappa |D|^\alpha(\nabla\rho^\pm)
  = -\nabla u^\pm_1\,\partial_1 \rho^\pm,\quad \nabla\rho^\pm|_{t=0}=\nabla\rho^\pm_0,
\end{equation*}
and $\partial_t(\nabla\rho^\pm)\in L^\infty([0,T^*[; L^\infty)$ which implies that
$\nabla\rho^\pm\in C([0,T^*[; L^\infty)$.

\subsection{Proof of Proposition \ref{prop FurP}-(4).}
Since $\theta^\pm\in C([0,T^*[;H^m(\mathbb{R}^2) )$ with $m>4$, then for every $t\in[0,T^*[$, there exists a constant $R_1>0$ (that may depend on $t$)
such that
\begin{equation*}
  \|\partial_1\rho^\pm\|_{L^\infty([0,t];L^\infty(B_{R_1}^c))}=\|\theta^\pm\|_{L^\infty([0,t];L^\infty(B_{R_1}^c)}
  \leq \|\partial_1\rho^\pm_0\|_{L^\infty}.
\end{equation*}

For $\partial_2\rho^\pm$, from \eqref{eq par^kRho}, and by denoting $f^\pm(t,x)=\partial_2(u^\pm_1 \,\partial_1\rho^\pm)(t,x)$, we infer that for every $t\in  ]0,T^*[$
and for some constant $R_2>0$ chosen later,
\begin{equation*}
  \|\partial_2\rho^\pm\|_{L^\infty([0,t]; L^\infty(B^c_{R_2}))}\leq \|\partial_2\rho^\pm_0\|_{L^\infty}
  +\int_0^t  \big\| e^{-\kappa(t-\tau)|D|^\alpha}f^\pm(\tau,\cdot)\big\|_{L^\infty(B^c_{R_2})}\mathrm{d}\tau .
\end{equation*}
Let $\chi$ be the cut-off function in the subsection \ref{subsec Pre1},
%satisfying that
and denote $\psi(x)\triangleq 1-\chi(x)$ for every $x\in\mathbb{R}^2$. Clearly, $\psi(x)\in C^\infty(\mathbb{R}^2)$ satisfies that
\begin{equation*}
  0\leq \psi\leq 1, \quad \mathrm{supp}\, \psi\subset B_1^c, \quad \psi\equiv 1 \,\;\mathrm{on}\,\; \overline{B}_{\frac{4}{3}}^c,
\end{equation*}
thus we get
\begin{equation*}
  \int_0^t \big\| e^{-\kappa(t-\tau)|D|^\alpha}f^\pm(\tau,\cdot)\big\|_{L^\infty(B^c_{R_2})} \mathrm{d}\tau \leq
  \int_0^t \big\| e^{-\kappa(t-\tau)|D|^\alpha}f^\pm(\tau,x)\, \psi\big(\frac{2x}{R_2}\big) \big\|_{L^\infty_x}\mathrm{d}\tau
  \triangleq \Gamma^\pm(t).
\end{equation*}
We divide it into several cases
\begin{equation*}
\begin{split}
  \Gamma^\pm(t)\leq & \int_0^t \Big\| e^{-\kappa(t-\tau)|D|^\alpha}\Big(f^\pm(\tau,\cdot)
  \psi(\frac{\cdot}{R_2/2})\Big)(x)\Big\|_{L^\infty_x} \mathrm{d}\tau+ \\
  & + \int_0^t \Big\|\Big( \big[e^{-\kappa(t-\tau)|D|^\alpha},\psi\big(\frac{\cdot}{R_2/2})\big]
  f^\pm (\tau,\cdot)\Big)(x)\Big\|_{L^\infty_x} \mathrm{d}\tau \\
  \triangleq & \, \Gamma^\pm_1(t) + \Gamma^\pm_2(t),
\end{split}
\end{equation*}
where $[X,Y]\triangleq XY-YX$ is the commutator.
For $\Gamma^\pm_1$, noticing that as $r\rightarrow \infty$,
\begin{equation*}
\begin{split}
  \|f^\pm\|_{L^\infty_t L^\infty(B^c_r)} \lesssim \|(\theta^+-\theta^-)\|_{L^\infty_t (H^m\cap L^p)}
  \big(\|\theta^\pm\|_{L^\infty_t (L^\infty(B^c_r))}+\|\partial_2\theta^\pm\|_{L^\infty_t (L^\infty(B^c_r))}\big)\longrightarrow 0,
\end{split}
\end{equation*}
we can choose $R_2$ large enough so that for every $t\in ]0,T^*[$,
\begin{equation*}
  \Gamma^\pm_1(t)\leq C t \|f^\pm \psi(2x/R_2)\|_{L^\infty_t L^\infty_x}
  \leq C t \|f^\pm\|_{L^\infty_t (L^\infty(B^c_{R_2/2}))}\leq \frac{\epsilon}{2}.
\end{equation*}
From \eqref{eq Exp2}, we can rewrite $\Gamma^\pm_2$ as
\begin{equation*}
\begin{split}
  \Gamma^\pm_2(t) & =\int_0^t \Big\|\int_{\mathbb{R}^2}K_\alpha( \kappa(t-\tau),y)f^\pm(\tau,x-y)
  \Big(\psi\big(\frac{x-y}{R_2/2}\big)-\psi\big(\frac{x}{R_2/2}\big)\Big)\mathrm{d}y\Big\|_{L^\infty_x} \mathrm{d}\tau \\
  & =\int_0^t \Big\|\int_{\mathbb{R}^2}K_\alpha(\kappa(t-\tau),y)f^\pm(\tau,x-y)
  \Big(\chi\big(\frac{x-y}{R_2/2}\big)-\chi\big(\frac{x}{R_2/2}\big)\Big)\mathrm{d}y\Big\|_{L^\infty_x} \mathrm{d}\tau.
\end{split}
\end{equation*}
Thus by using the estimate that
$$|g(z_1)-g(z_2)|\leq \|g\|_{C^{\alpha/2}(\mathbb{R}^2)} |z_1-z_2|^{\alpha/2}, \quad \forall z_1,z_2\in \mathbb{R}^2,$$
and the Minkowiski inequality, \eqref{eq FACT2}, \eqref{eq KerPro}, we obtain that
\begin{equation*}
\begin{aligned}
  \Gamma^\pm_2(t) & \leq \|\chi\|_{C^{\alpha/2}} \int_0^t \Big\|\int_{\mathbb{R}^2} K_\alpha(\kappa(t-\tau),y)
  \, |f^\pm(\tau,x-y)|\,\Big(\frac{|y|}{R_2/2}\Big)^{\alpha/2} \mathrm{d}y \Big\|_{L^\infty_x}\mathrm{d}\tau \\
  & \lesssim \|\chi\|_{C^{\alpha/2}} \|f^\pm\|_{L^\infty_t L^\infty_x}
  \int_0^t \int_{\mathbb{R}^2} \big(\kappa(t-\tau)\big)^{-\frac{2}{\alpha}}
  K_\alpha\Big( \frac{y}{(\kappa(t-\tau))^{1/\alpha}}\Big) \frac{|y|^{\alpha/2}}{R_2^{\alpha/2}}\mathrm{d}y \mathrm{d}\tau \\
  & \lesssim  \frac{\|\chi\|_{C^{\alpha/2}}}{R_2^{\alpha/2}} \|\theta^\pm\|_{L^\infty_t (H^m\cap L^p)}^2
  \int_0^t \big(\kappa (t-\tau) \big)^{\frac{1}{2}}\mathrm{d}\tau
  \int_{\mathbb{R}^2} K_\alpha(y) |y|^{\frac{\alpha}{2}}\mathrm{d}y \\
  & \lesssim R_2^{-\alpha/2} \|\chi\|_{C^{\alpha/2}} \|\theta^\pm\|_{L^\infty_t (H^m\cap L^p)}^2 (\kappa t)^{3/2}.
\end{aligned}
\end{equation*}
Thus through choosing $R_2$ large enough, we also have
\begin{equation*}
  \Gamma^\pm_2(t)\leq \frac{\epsilon}{2},\qquad \forall t\in]0,T^*[.
\end{equation*}
Denote $R=\max\{R_1,R_2\}$, then gathering the above estimates leads to \eqref{eq NRDecay}.

%%%%%%%%%%%%%%%%%%%%%%%%%%%%%%%%%%%%%%%%%%%%%%%%%%%%%%%%%%%%%%%%%%%%%%%%%%%%%%%%%%%%%%%%%%%%%%%%%%%%%%%%%%%%%%%%%%%
\section{Proof of Theorem \ref{thm glob}}\label{sec glob}
\setcounter{section}{5}\setcounter{equation}{0}

From Theorem \ref{thm local} and Proposition \ref{prop FurP},
we assume that $T^*>0$ is the maximal existence time of the solutions
$(\theta^+,\theta^-)\in C([0,T^*[; H^m\cap L^p)\cap L^\infty([0,T^*[; L^{\infty,1}_{x_2,x_1})\cap C^1([0,T^*[; H^{m_0})$
and $(\rho^+,\rho^-)\in L^\infty([0,T^*[;W^{2,\infty})\cap C([0,T^*[; W^{1,\infty})$ with $m>4$, $p\in]1,2[$ and $m_0=\min\{m-1,m-\alpha\}$.
There is also a blowup criterion: if $T^*<\infty$, we necessarily have
\begin{equation}\label{eq BlowC}
  \int_0^{T^*}\|(\theta^+,\theta^-)(t)\|_{L^\infty}\mathrm{d}t = \int_0^{T^*}\|(\partial_1\rho^+,\partial_1\rho^-)(t)\|_{L^\infty}\mathrm{d}t =\infty.
\end{equation}
We shall apply the nonlocal maximum principle method to the system \eqref{eq 1} to
show that some appropriate modulus of continuity is preserved, which implies that the Lipschitz
norm of $(\rho^+(t),\rho^-(t))$ is bounded uniformly in time. Clearly, this combined with \eqref{eq BlowC} leads to $T^*=\infty$.

Let $\lambda\in ]0,\infty[$ be a real number chosen later, $\omega$ be a stationary modulus of continuity
with its explicit formula shown later. According to the scaling transformation of \eqref{eq 1}, we set
\begin{equation}
  \omega_\lambda(\xi)\triangleq \lambda^{\alpha-1} \omega(\lambda \xi),\qquad \forall\, \xi\in [0,\infty[.
\end{equation}

First, we show that $(\rho^+_0,\rho^-_0)$ strictly obeys the MOC $\omega_\lambda$ for some $\lambda$.
From \eqref{eq Exp} and the non-negativity of $\theta$, we know that
$\|\rho^\pm_0\|_{L^\infty}\leq  \|\theta^\pm_0\|_{L^{\infty,1}_{x_2,x_1}}$. Denote $\omega_\lambda^{-1}$ and $\omega^{-1}$
the inverse functions of $\omega_\lambda$ and $\omega$
(if they are multi-valued for some $z$, we choose the smallest ones as their values), then we need
$\omega^{-1}_\lambda(3\|(\theta^+_0,\theta^-_0)\|_{L^{\infty,1}_{x_2,x_1}})<\infty$, so that for every $x,y$ satisfying
$|x-y|\geq \omega^{-1}_\lambda(3\|(\theta^+_0,\theta^-_0)\|_{L^{\infty,1}_{x_2,x_1}})$, we have
\begin{equation*}
  |\rho^\pm_0(x)-\rho^\pm_0(y)|\leq 2\|\theta^\pm_0\|_{L^{\infty,1}_{x_2,x_1}}\leq \frac{2}{3}\omega_\lambda(|x-y|).
\end{equation*}
For $\alpha\in]1,2]$, with no loss of generality we suppose that there are fixed constants $c_0, \xi_0>0$
depending on $\alpha$ such that
$\omega(\xi_0)= c_0$, which yields $\omega^{-1}(c_0)\leq \xi_0$.
Then we can choose some $\lambda\in ]0,\infty[$ such that
\begin{equation}\label{eq lamdCod1}
  \lambda^{\alpha-1} >\frac{3\|(\theta^+_0,\theta^-_0)\|_{L^{\infty,1}_{x_2,x_1}}}{c_0},
\end{equation}
and from $\omega_\lambda^{-1}(z)=\frac{1}{\lambda}\omega^{-1}(\frac{z}{\lambda^{\alpha-1}})$, we get
\begin{equation}\label{eq aaa}
  \omega^{-1}_{\lambda}(3\|(\theta_0^+,\theta_0^-)\|_{L^{\infty,1}_{x_2,x_1}})\leq \frac{\xi_0}{\lambda} < \infty.
\end{equation}
For $\alpha=1$, we have to call for that
$\omega$ is unbounded near infinity, so that $\omega^{-1}_\lambda(3\|(\theta^+_0,\theta^-_0)\|_{L^{\infty,1}_{x_2,x_1}})
=\lambda^{-1}\omega^{-1}(3\|(\theta^+_0,\theta^-_0)\|_{L^{\infty,1}_{x_2,x_1}})$ is meaningful for the large data. Thus for every $x,y$ satisfying
$\lambda|x-y|\geq \widetilde{C}_0$ with
\begin{equation*}
  \widetilde{C}_0 =
  \begin{cases}
    \xi_0, \quad  & \textrm{for} \;\,\alpha\in]1,2], \\
    \omega^{-1}(3\|(\theta^+_0,\theta^-_0)\|_{L^{\infty,1}_{x_2,x_1}}), \quad &  \textrm{for}\;\, \alpha=1,
  \end{cases}
\end{equation*}
we obtain
\begin{equation}
  |\rho^\pm_0(x)-\rho^\pm_0(y)|\leq \frac{2}{3}\omega_\lambda(|x-y|).
\end{equation}
The other treatment we can rely on is the mean value theorem, from which we have
\begin{equation*}
  |\rho^\pm_0(x)-\rho^\pm_0(y)|\leq \|\nabla\rho^\pm_0\|_{L^\infty} |x-y|.
\end{equation*}
Let $0<\delta_0<\widetilde{C}_0$. Due to the concavity of $\omega$, we infer that for every $x,y$ such that $\lambda|x-y|\leq \delta_0$,
\begin{equation*}
  \frac{\lambda^{\alpha-1} \omega(\delta_0)}{\delta_0}\leq \frac{\omega_\lambda(|x-y|)}{\lambda |x-y|}.
\end{equation*}
Thus by choosing $\lambda$ such that
\begin{equation}\label{eq lamdCod2}
  \lambda^\alpha > \frac{\delta_0}{\omega(\delta_0)} \|(\nabla\rho^+_0,\nabla\rho^-_0)\|_{L^\infty},
\end{equation}
we get that for every $x,y$ satisfying $x\neq y$ and $\lambda|x-y|\leq \delta_0$,
\begin{equation}
  |\rho^\pm_0(x)-\rho^\pm_0(y)|<\omega_\lambda(|x-y|).
\end{equation}
Finally, we consider the case of $x,y$ satisfying $\delta_0\leq \lambda|x-y|\leq \widetilde{C}_0$. Observe that $|\rho^\pm_0(x)-\rho^\pm_0(y)|\leq
\frac{\widetilde{C}_0}{\lambda}\|(\nabla\rho^+_0,\nabla\rho^-_0)\|_{L^\infty}$ and $\lambda^{\alpha-1}\omega(\delta_0)\leq \omega_\lambda(|x-y|)$.
Thus by choosing $\lambda$ satisfying
\begin{equation}\label{eq lamdCod3}
  \lambda^\alpha > \frac{\widetilde{C}_0}{\omega(\delta_0)} \|(\nabla\rho_0^+, \nabla \rho^-)\|_{L^\infty},
\end{equation}
we obtain that for every $x,y$ satisfying $\delta_0\leq \lambda|x-y|\leq \widetilde{C}_0$,
\begin{equation}
  |\rho^\pm_0(x)-\rho^\pm_0(y)|<\omega_\lambda(|x-y|).
\end{equation}
Hence, to fit our purpose, we can choose
\begin{equation}\label{eq lambda}
 \lambda\triangleq
 \begin{cases}
   \max\Big\{\Big(\frac{4\|(\theta^+_0,\theta^-_0)\|_{L^{\infty,1}_{x_2,x_1}}}{c_0}\Big)^{\frac{1}{\alpha-1}},
   \frac{\xi_0}{\|(\theta^+_0,\theta^-_0)\|_{L^{\infty,1}_{x_2,x_1}}}
   \|(\nabla\rho^+_0,\nabla\rho^-_0)\|_{L^\infty} \Big\},\quad &\alpha\in]1,2], \\
   \frac{\omega^{-1}\big(3\|(\theta^+_0,\theta^-_0)\|_{L^{\infty,1}_{x_2,x_1}}\big)}{\|(\theta^+_0,\theta^-_0)\|_{L^{\infty,1}_{x_2,x_1}}}
   \|(\nabla\rho^+_0,\nabla\rho^-_0)\|_{L^\infty},\quad & \alpha=1,
 \end{cases}
\end{equation}
and $\delta_0\triangleq \omega^{-1}\big(2\|(\theta^+_0,\theta^-_0)\|_{L^{\infty,1}_{x_2,x_1}}/\lambda^{\alpha-1}\big)$.

Let $T_*>0$ be the first time that the strict MOC $\omega_\lambda$ is lost by $\rho^\pm(t)$, i.e.,
\begin{equation}\label{eq T_*}
  T_*\triangleq \sup \{T\in [0,T^*[;  |\rho^\pm(t,x)-\rho^\pm(t,y)|< \omega_\lambda(|x-y|),\,\forall t\in [0,T[,\,
  \forall x\neq y\in\mathbb{R}^2 \}.
\end{equation}
Then we have the following assertion. %(with its proof postponed in the end of this section).
\begin{lemma}\label{lem scenario}
  Let $T_* >0$ be defined by \eqref{eq T_*}. Assume that $\omega$ moreover satisfies that
\begin{equation}\label{eq OmeC}
  \omega(0)=0, \quad \omega'(0)<\infty, \quad \omega''(0+)=-\infty.
\end{equation}
Then only three cases can occur:
\begin{enumerate}[(i)]
  \item $\rho^-$ strictly obeys the MOC $\omega_\lambda$ and there exist two separate points $x^+,y^+\in\mathbb{R}^2$ such that
  \begin{equation}\label{eq scena1}
    \rho^+(T_*,x^+)-\rho^+(T_*,y^+)=\omega_\lambda(\xi^+),\quad \textrm{with}\quad \xi^+=|x^+-y^+|;
  \end{equation}
  \item $\rho^+$ strictly obeys the MOC $\omega_\lambda$ and there exist two separate points $x^-,y^-\in\mathbb{R}^2$ such that
  \begin{equation}\label{eq scena2}
    \rho^-(T_*,x^-)-\rho^-(T_*,y^-)=\omega_\lambda(\xi^-),\quad \textrm{with}\quad \xi^-=|x^--y^-|;
  \end{equation}
  \item there exist four points $x^\pm,y^\pm\in\mathbb{R}^2$, $x^\pm\neq y^\pm$ such that
  \begin{equation}\label{eq scena3}
    \rho^\pm(T_*,x^\pm)-\rho^\pm(T_*,y^\pm)=\omega_\lambda(\xi^\pm),\quad \textrm{with}\quad \xi^\pm=|x^\pm-y^\pm|.
  \end{equation}
\end{enumerate}
Note that all $\xi^+$ and $\xi^-$ satisfy that $\xi^\pm \leq
\omega_\lambda^{-1}(3\|(\theta^+_0,\theta^-_0)\|_{L^{\infty,1}_{x_2,x_1}})$.
\end{lemma}

\begin{proof}[Proof of Lemma \ref{lem scenario}]
  It is clear to see that for every $t<T_*$, $\rho^\pm(t)$ strictly obeys the MOC $\omega_\lambda$, and from the time continuity
of $\rho^\pm(t)$, we have that for every $x,y\in\mathbb{R}^2$,
\begin{equation}\label{eq rhoFact}
  |\rho^\pm(T_*,x)-\rho^\pm(T_*,y)|\leq \omega_\lambda(|x-y|).
\end{equation}
Then for every
$x,y\in\mathbb{R}^2$, $x\neq y$, define
\begin{equation*}
  F^\pm(t,x,y)\triangleq \frac{|\rho^\pm(t,x)-\rho^\pm(t,y)|}{\omega_\lambda(|x-y|)},\qquad \forall t\in ]0,T^*[.
\end{equation*}
Obviously, $F^\pm(T_*,x,y)\leq 1$. We assume that $F^\pm(T_*,x,y)<1$ for all $x\neq y\in\mathbb{R}^2$, since otherwise
the claim follows.

First, denote $\overline{C}_0\triangleq\omega_\lambda^{-1}\big(3\|(\theta^+_0,\theta^-_0)\|_{L^{\infty,1}_{x_2,x_1}}\big)
=\lambda^{-1}\omega^{-1}\big(3\|(\theta^+_0,\theta^-_0)\|_{L^{\infty,1}_{x_2,x_1}}/\lambda^{\alpha-1}\big)$,
and we find that for every $x,y$ satisfying $|x-y|\geq \overline{C}_0$,
\begin{equation*}
  2\|\theta^\pm_0\|_{L^{\infty,1}_{x_2,x_1}}\leq \frac{2}{3}\omega_\lambda(\overline{C}_0)\leq \frac{2}{3}\omega_\lambda(|x-y|).
\end{equation*}
Thus by \eqref{eq rhoLinf}, we have for every $t\in ]0,T^*[$ and $x,y$ satisfying $|x-y|\geq \overline{C}_0$,
\begin{equation}
  |\rho^\pm(t,x)-\rho^\pm(t,y)|\leq 2\|\theta^\pm_0\|_{L^{\infty,1}_{x_2,x_1}} \leq \frac{2}{3} \omega_\lambda(|x-y|).
\end{equation}
Second, we consider the case of $x,y$ near infinity. From the mean value theorem, we get
for every $t\in]0,T^*[$ and for every $x,y$ satisfying that
$0<|x-y|\leq\overline{C}_0$ and $x$ or $y$ belongs to $B^c_{R+\overline{C}_0}$ with $R>0$ fixed later,
\begin{equation*}
  |\rho^\pm(t,x)-\rho^\pm(t,y)|\leq \|\nabla\rho^\pm\|_{L^\infty([0,t]; L^\infty(B^c_R))} |x-y|.
\end{equation*}
By the concavity of $\omega$ and $|x-y|\leq \overline{C}_0$, we find that
\begin{equation*}
  \lambda\,\frac{3\|(\theta^+_0,\theta^-_0)\|_{L^{\infty,1}_{x_2,x_1}}}
  {\omega^{-1}(3\|(\theta^+_0,\theta^-_0)\|_{L^{\infty,1}_{x_2,x_1}}/\lambda^{\alpha-1})}
  =\frac{\omega_\lambda(\overline{C}_0)}{\overline{C}_0} \leq \frac{\omega_\lambda(|x-y|)}{|x-y|}.
\end{equation*}
In order to make
\begin{equation*}
  \|\nabla\rho^\pm\|_{L^\infty([0,t]; L^\infty(B^c_R))}\leq \frac{\lambda}{2}
  \frac{3\|(\theta^+_0,\theta^-_0)\|_{L^{\infty,1}_{x_2,x_1}}}
  {\omega^{-1}(3\|(\theta^+_0,\theta^-_0)\|_{L^{\infty,1}_{x_2,x_1}}/\lambda^{\alpha-1})},
\end{equation*}
from $\lambda\geq \frac{\widetilde{C}_0}{\|(\theta_0^+,\theta_0^-)\|_{L^{\infty,1}_{x_2,x_1}}}\|(\nabla\rho^+_0,\nabla\rho^-_0)\|_{L^\infty}$
and \eqref{eq lamdCod1}, it suffices to choose $R$ such that
\begin{equation}\label{eq key1}
  \|\nabla\rho^\pm\|_{L^\infty([0,t]; L^\infty(B^c_R))}\leq \frac{3}{2}\|(\nabla\rho^+_0,\nabla\rho^-_0)\|_{L^\infty}.
\end{equation}
This estimate can be guaranteed by \eqref{eq NRDecay}, and we denote the chosen number by $R(t)$.
Thus we obtain that for every
$x,y$ satisfying that $0<|x-y|\leq\overline{C}_0$ and $x$ or $y$ belongs to $B^c_{R(t)+\overline{C}_0}$,
\begin{equation*}
  |\rho^\pm(t,x)-\rho^\pm(t,y)|\leq \frac{1}{2}\omega_\lambda(|x-y|), \quad \forall t\in ]0,T^*[.
\end{equation*}
In particular, there exists a number $h_1>0$ such that for every $x,y$ satisfying that $0<|x-y|\leq\overline{C}_0$
and $x$ or $y$ belongs to $B^c_{R(T_*+h_1)+\overline{C}_0}$,
\begin{equation}
  |\rho^\pm(t,x)-\rho^\pm(t,y)|\leq \frac{1}{2}\omega_\lambda(|x-y|), \quad \forall t\in [T_*,T_*+h_1],
\end{equation}
Next we reduce to consider the case that $x,y\in B_{R(T_*+h_1)+\overline{C}_0}$ and $0<|x-y|\leq \overline{C}_0$.
Since \eqref{eq OmeC} and $\rho^\pm(T_*)\in W^{2,\infty}$, from Lemma \ref{lem MudLem} we get that
\begin{equation*}
  \|\nabla\rho^\pm(T_*)\|_{L^\infty(B_{R(T_*+h_1)+\overline{C}_0})}< \omega_\lambda'(0)= \lambda^\alpha \omega'(0).
\end{equation*}
From $\rho^\pm\in C([0,T^*[; W^{1,\infty})$, there exists small constants $h_2,\tilde{\delta}>0$ such that
for every $t\in[T_*,T_*+h_2]$,
\begin{equation*}
  \|\nabla\rho^\pm(t)\|_{L^\infty(B_{R(T_*+h_1)+\overline{C}_0})}\leq
  (1-\tilde{\delta})\lambda^\alpha \frac{\omega(\tilde{\delta})}{\tilde{\delta}}.
\end{equation*}
Thus for every $x,y\in B_{R(T_*+h_1)+\overline{C}_0}$ satisfying $0<\lambda|x-y|\leq \tilde{\delta}$, from that
\begin{equation*}
  \frac{\lambda^{\alpha-1}\omega(\tilde\delta)}{\tilde\delta}\leq \frac{\omega_\lambda(|x-y|)}{\lambda|x-y|},
\end{equation*}
we obtain
\begin{equation}
\begin{split}
  |\rho^\pm(t,x)-\rho^\pm(t,y)| & \leq \|\nabla\rho^\pm(t)\|_{L^\infty(B_{R(T_*+h_1)+\overline{C}_0})}|x-y| \\
  & \leq (1-\tilde\delta)\omega_\lambda(|x-y|), \qquad \quad\forall t\in[T_*,T_*+h_2].
\end{split}
\end{equation}
Now it remains to treat the case that the continuous function $F^\pm(t,x,y)$ on the compact set
\begin{equation*}
  \mathcal{K}:= \{(x,y)\in\mathbb{R}^2\times\mathbb{R}^2; \max\{|x|,|y|\}
  \leq R(T_*+h_1)+\overline{C}_0,\,\,|x-y|\geq \tilde\delta/\lambda \}.
\end{equation*}
By virtue of $F^\pm(T_*,x,y)<1$ for all $(x,y)\in\mathcal{K}$, we have that there exist small constants $h_3,\bar\delta>0$ such that
\begin{equation}
  F^\pm(t,x,y)\leq 1-\bar\delta, \qquad \forall \, (t,x,y)\in [T_*,T_*+h_3]\times \mathcal{K}.
\end{equation}
Set $h\triangleq \min\{h_1,h_2,h_3 \}>0$, then by gathering the above estimates, we know that $\rho^\pm(T_*+h)$
strictly obeys the MOC $\omega_\lambda$ and this clearly contradicts with the definition of $T_*$.

\end{proof}

Now we shall show that this scenarios $(i)$-$(iii)$ can not happen. More precisely, we shall prove
\begin{equation}\label{eq REDUCT}
\begin{cases}
  \textrm{for}\; (i),&\quad (f^+)'(T_*)<0,\qquad \textrm{with}\;\; f^+(T_*)=\rho^+(T_*,x^+)-\rho^+(T_*,y^+), \\
  \textrm{for}\; (ii),&\quad (f^-)'(T_*)<0, \qquad \textrm{with}\;\; f^-(T_*)=\rho^-(T_*,x^-)-\rho^-(T_*,y^-), \\
  \textrm{for}\; (iii),&\quad (f^\pm)'(T_*)<0, \qquad \textrm{with}\;\; f^\pm(T_*)=\rho^\pm(T_*,x^\pm)-\rho^\pm(T_*,y^\pm).
\end{cases}
\end{equation}
Clearly, this means that for some $t<T_*$, the strict MOC $\omega_\lambda$ is lost by $\rho^+(t)$ or $\rho^-(t)$,
and this contradicts the definition of $T_*$.

Since $\rho^\pm$ solves the equation \eqref{eq 1} in the classical pointwise sense, we directly have
\begin{equation*}
\begin{split}
  \partial_t\rho^\pm(T_*,x^\pm)-\partial_t\rho^\pm(T_*,y^\pm)
  = & - u^\pm \cdot\nabla\rho^\pm(T_*,x^\pm)+ u^\pm\cdot\nabla\rho^\pm(T_*,y^\pm) + \\
  & + [- |D|^\alpha]\rho^\pm(T_*,x^\pm)- [- |D|^\alpha]\rho^\pm(T_*,y^\pm)
\end{split}
\end{equation*}
with
\begin{equation*}
  u^\pm=\pm(\mathcal{R}_1^2 \mathcal{R}_2^2(\rho^+-\rho^-),0).
\end{equation*}
Taking advantage of Lemma \ref{lem MudRR}, \ref{lem MudDiss} and the change of variable,
we find that
\begin{equation*}
\begin{cases}
  \textrm{for}\; (i),&\quad (f^+)'(T_*)\leq \lambda^{2\alpha-1}(\Omega \omega'+ \Psi_\alpha)(\lambda \xi^+), \\
  \textrm{for}\; (ii), &\quad (f^-)'(T_*)\leq \lambda^{2\alpha-1}(\Omega \omega' + \Psi_\alpha)(\lambda \xi^-), \\
  \textrm{for}\; (iii), &\quad (f^\pm)'(T_*) \leq \lambda^{2\alpha-1}(\Omega\omega' +\Psi_\alpha)(\lambda \xi^\pm),
\end{cases}
\end{equation*}
with $\Omega$ and $\Psi_\alpha$ defined by \eqref{eq Omega} and \eqref{eq Psi} respectively.

Next we construct appropriate moduli of continuity satisfying \eqref{eq OmeC} in the spirit of \cite{KisNV}.
Let $0<\gamma<\delta<1$ be two absolute constants chosen later, then for every $\alpha\in [1,2]$,
we define the following continuous functions that for $\alpha=1$
\begin{equation}\label{eq MOC}
\mathrm{MOC}_1\,\begin{cases}
  \omega(\xi)=\, \xi-\xi^{3/2},  \quad & \textrm{for}\;\, \xi \in [0,\delta],\\
  \omega'(\xi)=\frac{\gamma}{\xi(4+\log(\xi/\delta))},\quad & \textrm{for}\;\; \xi\in ]\delta,\infty[,
\end{cases}
\end{equation}
and for $\alpha\in ]1,2]$
\begin{equation}\label{eq MOCal}
\mathrm{MOC}_\alpha\,\begin{cases}
  \omega(\xi)=\, \xi-\xi^{3/2},  \quad & \textrm{for}\;\, \xi \in [0,\delta],\\
  \omega'(\xi)=0,\quad & \textrm{for}\;\; \xi\in ]\delta,\infty[,
\end{cases}
\end{equation}
Notice that, for small $\delta$, we have $\omega'(\delta-)\approx 1$, while $\omega'(\delta+)\leq \frac{1}{4}$, thus
$\omega$ is a concave piecewise $C^2$ function if $\delta$ is small enough. Obviously, $\omega(0)=0$,
$\omega'(0)=1$ and $\omega''(0+)=-\infty$. For $\alpha=1$, $\omega$ is unbounded near infinity, and for $\alpha\in ]1,2]$,
$\omega$ is a bounded function with maximum $\delta-\delta^{3/2}$ (in \eqref{eq lambda}, we can choose
$c_0=\delta-\delta^{3/2}$ and $\xi_0=\delta$).

Then our target is to prove that for suitable MOC given by \eqref{eq MOC} and \eqref{eq MOCal},
\begin{equation}\label{eq TARGET}
  \Omega(\xi) \omega'(\xi) + \Psi_\alpha(\xi)<0,
\end{equation}
for all $0<\xi\leq \lambda\overline{C}_0 =\lambda\omega_\lambda^{-1}(3\|(\theta_0^+,\theta^-_0)\|_{L^{\infty,1}_{x_2,x_1}})$,
more precisely,
\begin{equation*}
\begin{split}
  \Big(A_1 \omega(\xi)  + A_2\int_0^\xi \frac{\omega(\eta)}{\eta}\mathrm{d}\eta
  + A_2 \xi \int_\xi^\infty \frac{\omega(\eta)}{\eta^2}\mathrm{d}\eta \Big)\omega'(\xi)
  + \Psi_\alpha(\xi) <0, \qquad \forall \,\xi\in ]0, \lambda\overline{C}_0].
\end{split}
\end{equation*}
Note that from \eqref{eq aaa}, we know $\lambda\overline{C}_0\leq \delta$ for $\alpha\in ]1,2]$.

We divide into two cases.
\vskip0.1cm
Case 1: $\alpha\in[1,2]$ and $0<\xi\leq \delta$.
\vskip0.1cm
Since $\frac{\omega(\eta)}{\eta}\leq \omega'(0)=1$ for all $\eta>0$, we have
$\int_0^\xi \frac{\omega(\eta)}{\eta}\textrm{d}\eta \leq \xi$
and $\int_\xi^\delta \frac{\omega(\eta)}{\eta^2}\textrm{d}\eta \leq \int_\xi^\delta \frac{1}{\eta}\textrm{d}\eta= \log(\delta/\xi).$
Further,
\begin{equation*}
  \begin{cases}
  \int_\delta^\infty \frac{\omega(\eta)}{\eta^2}\textrm{d}\eta= \frac{\omega(\delta)}{\delta}+
  \int_\delta^\infty\frac{\gamma}{\eta^{2}(4+\log(\eta/\delta))}\textrm{d} \eta
  \leq 1+ \frac{\gamma}{4 \delta } \leq 2, \quad &\mathrm{for}\;\,\alpha=1, \\
  \int_\delta^\infty \frac{\omega(\eta)}{\eta^2}\textrm{d}\eta\leq \int_\delta^\infty \frac{\delta}{\eta^2} \mathrm{d}\eta = 1.
  \quad &\mathrm{for}\;\, \alpha \in ]1,2],
\end{cases}
\end{equation*}
Obviously $\omega'(\xi)\leq\omega'(0)=1$, so we get that the positive
part is bounded by $\xi(A_1+3A_2+ A_2\log(\delta/\xi))$.

For the negative part, we have $\omega''(\xi)=-\frac{3}{4}\xi^{-\frac{1}{2}}<0$, and
\begin{equation*}
\begin{split}
 \int_0^{\frac{\xi}{2}}\frac{\omega(\xi+2\eta)+\omega(\xi-2\eta)-2\omega(\xi)}{\eta^{1+\alpha}}\textrm{d}
 \eta \leq \int_0^{\frac{\xi}{2}}\frac{\omega''(\xi)2\eta^{2}}{\eta^{1+\alpha}}\textrm{d}\eta
 \leq -\frac{3}{4}\xi^{\frac{3}{2}-\alpha},\quad \mathrm{for}\;\,\alpha\in [1,2[.
\end{split}
\end{equation*}

Hence by choosing $\delta$ small enough, we have for all $\xi\in]0,\delta],$
\begin{equation*}
\begin{cases}
  \xi\big( A_1+3A_2 +A_2\log(\delta/\xi) -\frac{3 B_\alpha}{4} \xi^{\frac{1}{2}-\alpha}\big)<0,\quad &  \mathrm{for}\;\,\alpha\in [1,2[, \\
  \xi\big( A_1+2A_2 +A_2\log(\delta/\xi) -\frac{3 }{4} \xi^{-\frac{3}{2}}\big)<0,\quad &  \mathrm{for}\;\,\alpha=2.
\end{cases}
\end{equation*}

\vskip0.1cm
Case 2: $\alpha=1$ and $\xi\geq \delta$.
\vskip0.1cm

To show \eqref{eq TARGET}, this is almost identical to the corresponding part of \cite{KisNV}, and it suffices to choose $\gamma$
small enough; we here omit the details.

Therefore, \eqref{eq REDUCT} holds, and it implies that $T_*=T^*$. Moreover, for every $t\in [0,T^*[$, we have
$\|\nabla\rho^\pm(t)\|_{L^\infty}\leq \omega'_\lambda(0)=\lambda^\alpha$ with $\lambda$ defined by \eqref{eq lambda}.
This estimate combining with the breakdown criterion \eqref{eq BlowC} yields that $T^*=\infty$.

\section{Appendix}
\setcounter{section}{6}\setcounter{equation}{0}

In this section, we consider the Groma-Balogh model with generalized dissipation
\begin{equation}\label{eq App1}
\begin{cases}
  \partial_t \rho^+ + u \cdot\nabla \rho^+ +  |D|^\alpha \rho^+ =0, \quad \alpha\in ]0,2], \\
  \partial_t \rho^- - u \cdot\nabla \rho^- +  |D|^\beta \rho^- = 0, \quad \beta\in ]0,2],\\
  u=\big(\mathcal{R}_1^2 \mathcal{R}_2^2 (\rho^+ -\rho^-),0\big), \\
  \rho^+|_{t=0}=\rho^+_0,\quad \rho^-|_{t=0}=\rho^-_0.
\end{cases}
\end{equation}
In terms of the dislocation densities $\theta^\pm\triangleq \partial_1\rho^\pm$, we write
\begin{equation}\label{eq App2}
\begin{cases}
  \partial_t \theta^+ + \partial_1(u_1\, \theta^+)
  +  |D|^\alpha \theta^+ =0, \quad \alpha\in ]0,2], \\
  \partial_t \theta^- - \partial_1 (u_1\, \theta^-)
  +  |D|^\beta \theta^- = 0,\quad \beta\in ]0,2], \\
  u_1=\mathcal{R}_1 \mathcal{R}_2^2 |D|^{-1}(\theta^+-\theta^-), \\
  \theta^+|_{t=0}=\theta^+_0,\quad \theta^-|_{t=0}=\theta^-_0.
\end{cases}
\end{equation}

Similarly as Theorem \ref{thm glob}, we get the following global result in the subcritical regime.
\begin{proposition}\label{prop glob-g}
  Let $(\alpha,\beta)\in ]1,2]^2$, $\alpha\neq \beta$, $(\theta^+_0,\theta^-_0)\in H^m(\mathbb{R}^2)\cap L^p(\mathbb{R}^2) \cap L^{\infty,1}_{x_2,x_1}(\mathbb{R}^2)$ with $m>4$, $p\in ]1,2[$
be composed of non-negative real scalar functions. Assume that $\rho^\pm_0(x_1,x_2)=\int_{-\infty}^{x_1}\theta^\pm_0(\tilde{x}_1,x_2)\mathrm{d}
\tilde{x}_1$ satisfy that for each $k=1,2,3$, $\partial_2^k \rho^\pm_0\in L_x^\infty(\mathbb{R}^2)$ and
$\lim_{x_1\rightarrow-\infty}\partial_2^k \rho^\pm_0(x)=0$ for every $x_2\in\mathbb{R}$.
Then there exists a unique global solution
\begin{equation*}
  (\theta^+,\theta^-)\in C([0,\infty[;H^m\cap L^p)\cap L^\infty([0,\infty[;L^{\infty,1}_{x_2,x_1})
\end{equation*}
to the equation \eqref{eq App2}. Moreover, $(\rho^+,\rho^-)\in L^\infty([0,\infty[; W^{3,\infty})\cap C([0,\infty[;W^{1,\infty})$ solves
the equation \eqref{eq App1} in the classical pointwise sense.
\end{proposition}

\begin{remark}
  When $1=\alpha<\beta\leq 2$ or $1=\beta<\alpha\leq 2$, in a similar way we can obtain the same global result
under the condition that the norm $\|(\theta^+_0,\theta^-_0)\|_{L^{\infty,1}_{x_2,x_1}}$ is small enough.
\end{remark}

\begin{proof}[Proof of Proposition \ref{prop glob-g}]
  Note that Theorem \ref{thm local} and Proposition \ref{prop FurP} also hold for the systems \eqref{eq App1}
and \eqref{eq App2}, and it remains to show that for every $T\in ]0,T^*[$, there is an upper bound of the quantity
$\int_0^T \|(\partial_1\rho^+,\partial_1\rho^-)(t)\|_{L^\infty}\mathrm{d}t$.

With no loss of generality, we fix $1<\alpha<\beta\leq 2$ in the sequel. Let $\omega$ be an appropriate MOC chosen later, and denote
\begin{equation*}
  \omega_\lambda(\xi)=\lambda^{\alpha-1}\omega(\lambda \xi),\qquad \forall \xi>0.
\end{equation*}
Let $\lambda\geq 1$ be defined by \eqref{eq lambda} with $\alpha\in]1,2]$ (if the quantity in \eqref{eq lambda} is less than 1,
set $\lambda=1$), similarly as in Section \ref{sec glob}, we get $\rho^\pm_0$ strictly satisfy the MOC $\omega_\lambda$.
Let $T_*$ be defined by \eqref{eq T_*}, we also find that Lemma \ref{lem scenario} holds true, and it suffices to show that
\eqref{eq REDUCT} is satisfied.

From \eqref{eq App1} we have
\begin{equation*}
  (f^+)'(T_*) =  - u \cdot\nabla\rho^+(T_*,x^+)+ u\cdot\nabla\rho^+(T_*,y^+) +
  [- |D|^\alpha]\rho^+(T_*,x^+)- [- |D|^\alpha]\rho^+(T_*,y^+),
\end{equation*}
and
\begin{equation*}
  (f^-)'(T_*) =   u \cdot\nabla\rho^-(T_*,x^-)- u \cdot\nabla\rho^-(T_*,y^-) +
  [- |D|^\beta]\rho^-(T_*,x^-)- [- |D|^\beta]\rho^-(T_*,y^-),
\end{equation*}
with $u =(\mathcal{R}_1^2 \mathcal{R}_2^2(\rho^+-\rho^-),0)$.
By Lemma \ref{lem MudRR}, \ref{lem MudDiss} and the change of variable,
we obtain that
\begin{equation*}
\begin{cases}
  (f^+)'(T_*)&\leq \lambda^{2\alpha-1}\big(\Omega\omega'
  + \Psi_\alpha\big)(\lambda \xi^+),\qquad\qquad \textrm{for}\; (i),(iii),  \\
   (f^-)'(T_*)&\leq \lambda^{2\alpha-1}\big(\Omega \omega'
  + \lambda^{\beta-\alpha} \Psi_\beta\big)(\lambda \xi^-) \\
  & \leq \lambda^{2\alpha-1}\big(\Omega \omega'  + \Psi_\beta\big)(\lambda \xi^-),\qquad\qquad \textrm{for}\; (ii),(iii),
\end{cases}
\end{equation*}
where $\Omega$ is defined by \eqref{eq Omega} corresponding to $\omega$, and
$\Psi_\alpha$, $\Psi_\beta$ are defined by \eqref{eq Psi}.

Next we construct suitable modulus of continuity satisfying \eqref{eq OmeC}.
Let $0<\delta<1$ be a fixed constant chosen later, then for every $1<\alpha<\beta\leq 2$,
we define the following continuous function
\begin{equation}\label{eq MOC2}
\mathrm{MOC}\,\begin{cases}
  \omega(\xi) =\, \xi-\xi^{3/2},  \quad & \textrm{for}\;\, \xi \in [0,\delta],\\
  \omega'(\xi)=0, \quad & \textrm{for}\;\; \xi\in ]\delta,\infty[.
\end{cases}
\end{equation}

Then our target is to prove that for the suitable MOC given by \eqref{eq MOC2},
\begin{equation}\label{eq TARGET1}
  \Omega(\xi)\omega'(\xi) + \Psi_\alpha(\xi)<0,\qquad \forall \xi\in ]0,\delta],
\end{equation}
and
\begin{equation}\label{eq TARGET2}
  \Omega(\xi) \omega'(\xi) + \Psi_\beta(\xi)<0,\qquad \forall \xi\in ]0,\delta].
\end{equation}
Similarly as proving \eqref{eq TARGET}, for appropriate positive constants $\delta$ that may depend on $\alpha,\beta$,
we can show \eqref{eq TARGET1} and \eqref{eq TARGET2} are satisfied.

Therefore, we have $T_*=T^*$. Moreover, for every $t\in [0,T^*[$, we have
$\|\nabla\rho^\pm(t)\|_{L^\infty}\leq \lambda^\alpha$.
This combining with the breakdown criterion \eqref{eq BlowC} yields $T^*=\infty$.
\end{proof}

\textbf{Acknowledgements:} The authors would like to express their
deep gratitude to Prof. R\'egis Monneau for helpful discussion and
valuable suggestions on subcritical dissipative cases. Dong Li  was
supported in part by NSF Grant 0908032,  Changxing Miao  and Liutang
Xue  were partially supported by the NSF of China under grants
11171033 and No. 11101042.

%%%%%%%%%%%%%%%%%%%%%%%%%%%%%%%%%%%%%%%%%%%%%%%%%%%%%%%%%%%%%%%%%%%%%%%%%%%%%%%%%%%%%%%%%%%%%%%%%%%%%%%%%%%%%%%%%%%%%
%references


\begin{thebibliography}{60}

\bibitem{AHLM}O. Alvarez, P. Hoch, Y. Le Bouar and R. Monneau, Dislocation dynamics: short time existence and uniqueness of the solution,
  Arch. Rati. Mech. Anal., \textbf{181} (3) (2006), 449-504.
\bibitem{AbidiH}H. Abidi, T. Hmidi, On the global wellposedness of the critical quasi-geostrophic equation,
  SIAM J. Math. Anal., \textbf{40}(2008), 167-185.
\bibitem{CV}L. Caffarelli and V. Vasseur, Drift diffusion equations with fractional diffusion and the quasi-geostrophic equations.
  Annals of Math., \textbf{171} Issue 3 (2010), 1903-1930.
\bibitem{BCLM}G. Barles, P. Cardaliaguet, O. Ley and R. Monneau, General existence results and uniqueness for dislocation equations,
  SIAM J. Math. Anal., \textbf{40} (1) (2008), 44-69.
\bibitem{CEMR08}M. Cannone, A. El Hajj, R. Monneau and F. Ribaud, Global existence for a system of non-linear and non-local
    transport equations describing the dynamics of dislocation densities. Arch. Rati. Mech. Anal., \textbf{196} (2010), 71-96.
\bibitem{ChenMZ}Q. Chen, C. Miao and Z. Zhang, A new Bernstein's inequality and the 2D dissipative quasi-geostrophic equation.
   Comm. Math. Phys. \textbf{271}(2007), 821-838.
\bibitem{ConMT}P. Constantin, A.J. Majda and E. Tabak, Formation of strong fronts in the 2-D quasigeostrophic thermal active scalar.
   Nonlinearity, \textbf{7}(1994), 1495-1533.
\bibitem{ConV}P. Constantin and V. Vicol, Nonlinear maximum principles for dissipative linear nonlocal operators and applications.
ArXiv:math.AP/1110.0179v1, to appear in GAFA.
\bibitem{ConWu99}P. Constantin and J. Wu, Behavior of solutions of 2D quasi-geostrophic equations. SIAM J. Math. Anal., \textbf{30}(1999), 937-948.
\bibitem{CorC}A. C\'ordoba and D. C\'ordoba, A maximum principle applied to the quasi-geostrophic equations.
   Comm. Math. Phys., \textbf{249}(2004), 511-528.
\bibitem{Dab}M. Dabkowski, Eventual regularity of the solutions to the supercritical dissipative quasi-geostrophic equation,
   Geom. Funct. Anal. \textbf{21} (2011), no. 1, 1--13.
\bibitem{DongD}H. Dong and D. Du,  Global well-posedness and a decay estimate for the critical dissipative quasi-geostrophic equation in
the whole space. Disc. Cont. Dyna. Syst., \textbf{21} no. 4 (2008), 1095-1101.
\bibitem{DroI}J. Droniou and C. Imbert, Fractal first order partial differential equations. Arch. Rati. Mech. Anal.,\textbf{182} (2) (2006), 299-331.
\bibitem{Duo}J. Duoandikoetxea. {\em Fourier Analysis}. Translated and revised by D. Cruz-Uribe. GSM vol.\textbf{29}, AMS, 2000.
\bibitem{Elha}A. El Hajj, Short time existence and uniqueness in H\"older spaces for the 2D dynamics of dislocation densities.
  Ann. Inst. H. Poincar\'e Anal. Non Lin\'eaire, \textbf{27} (2010) 21-35.
\bibitem{Gro97}I. Groma, Link between the microscopic and mesoscopic length-scale description of the collective behaviour of
    dislocations. Phys. Rev. B, \textbf{56} (1997), 5807.
\bibitem{GroBal99}I. Groma and P. Balogh, Investigation of dislocation pattern formation in a two-dimensional self-consistent
    field approximation. Acta Mater., \textbf{47} (1999), 3647-3654.
\bibitem{HmdK}T. Hmidi and S. Keraani, Global solutions of the supercritical 2D dissipative quasi-geostrophic equation,
  Adv. Math., \textbf{214} (2007), 618-638.
\bibitem{HiLo}J. Hirth and J. Lothe, \textit{Theory of dislocations}. Second edition. Wiley, 1982.
\bibitem{Ju}N. Ju, The maximum principle and the global attractor for the dissipative 2D quasi-geostrophic equations.
    Comm. Math. Phys., \textbf{255} (2005), 161-181.
\bibitem{KisN}A. Kiselev and F. Nazarov, A variation on a theme of Caffarelli and Vasseur.
Zap. Nauchn. Sem. POMI, \textbf{370}(2010), 58-72.
\bibitem{KisNV}A. Kiselev, F. Nazarov and A. Volberg, Global well-posedness for the critical 2D dissipative quasi-geostrophic equation,
  Invent. Math. \textbf{167}(2007), 445-453.
\bibitem{Kis}A. Kiselev, Nonlocal maximum principle for active scalars. Adv. in Math., \textbf{227} no. 5 (2011), 1806-1826.
\bibitem{LiRZ}D. Li, J. Rodrigo and X. Zhang, Exploding solutions for a nonlocal quadratic evolution problem.
         Rev. Mat. Iberoam. \textbf{26} (1) (2010), 295-332.
\bibitem{MiaoX}C. Miao and L. Xue, On the regularity of a class of generalized quasi-geostrophic equations.
   J. Differential Equations, \textbf{251} (2011), 2789-2821.
%\bibitem{MiaoX2}C. Miao and L. Xue, Global well-posedness for a modified critical dissipative quasi-geostrophic equation.
%   J. Differential Equations, \textbf{252}, (2012), 792-818.
\bibitem{Res}S. Resnick, Dynamical problems in nonlinear advective partial differential equations,
   Ph.D. thesis, University of Chicago, 1995.
\bibitem{Silv}L. Silvestre, Holder estimates for advection fractional-diffusion equations.
  Ann. Scuola Norm. Sup. Pisa Cl. Sci., Accepted for publication.
%\bibitem{Silv2}L. Silvestre, On the differentiability of the solution to an equation with drift and fractional diffusion.
%  Indiana Univ. Math. J., Accepted for publication.
\end{thebibliography}
\end{document}